\documentclass[11pt]{amsart}
\usepackage{hyperref}

\usepackage[utf8]{inputenc}

\usepackage{amsmath}
\usepackage{bm}

\usepackage{enumitem}
\usepackage{amssymb} 
\usepackage{esint}
 \usepackage{cite}
\usepackage{xcolor}
\usepackage[top=1in, left=1in, right=1in, bottom=1in, marginpar=2cm]{geometry}

\usepackage[normalem]{ulem} 
\usepackage{cancel}

\newtheorem{theorem}{Theorem}[section]
\newtheorem{lemma}[theorem]{Lemma}
\newtheorem{proposition}[theorem]{Proposition}
\newtheorem{corollary}[theorem]{Corollary}
\theoremstyle{definition}
\newtheorem{definition}[theorem]{Definition}
\theoremstyle{remark}

\numberwithin{equation}{section}
\theoremstyle{claim}

\newcommand{\R}{{\mathbb R}}
\newcommand{\Z}{{\mathbb Z}}

\newcommand{\N}{{\mathbb N}}

\def\Z{{\mathbb Z}}
\def\R{{\mathbb R}}

\def\N{{\mathbb N}}

\def\a{\alpha}
\def\b{\beta}
\def\e{\varepsilon}
\def\d{\delta}

\def\n{\nabla}
\def\p{\partial}

\def\W{\Omega}
\def\g{\gamma}

\def\1{\left(}
\def\2{\right)}
\def\3{\left\{}
\def\4{\right\}}
\def\8{\infty}

\def\ss{\subseteq}
\def\cc{\subset\subset}

\def\Comp{H_{\leq 1}}

\title[Semilinear Obstacle Problem]{The Free Boundary  for semilinear problems \\ with highly oscillating singular terms}

\author{Mark Allen}
\address{Department of Mathematics, Brigham Young University, Provo,  UT 84602, US}
\email{allen@mathematics.byu.edu}

\author{Dennis Kriventsov}
\address{Department of Mathematics, Rutgers University,  Piscataway, NJ}
\email{dnk34@math.rutgers.edu}

\author{Henrik Shahgholian}
\address{Department of Mathematics, KTH  Royal Institute of Technology, Stockholm, Sweden}
\email{henriksh@kth.se}

\begin{document}

\begin{abstract}
 
 We investigate general semilinear (obstacle-like)  problems of the form $\Delta u = f(u)$, where $f(u)$ has a singularity/jump at  $\{u=0\}$ giving rise to a free boundary. Unlike many works on such equations where $f$ is approximately homogeneous near $\{u = 0\}$, we work under assumptions allowing for highly oscillatory behavior.

We establish  the $C^\infty$ regularity of the free boundary $\partial \{u>0\}$ at flat points. Our approach is to first establish that flat free boundaries are Lipschitz, using a comparison argument with the Kelvin transform. For higher regularity, we study the highly degenerate PDE satisfied by ratios of derivatives of $u$, using changes of variable and then the hodograph transform. Along the way, we prove and make use of new Caffarelli-Peral type $W^{1, p}$ estimates for such degenerate equations. Much of our approach appears new even in the case of Alt-Phillips and classical obstacle problems.
\end{abstract}

\maketitle

\tableofcontents

\section{Introduction}

In this paper, we consider one-phase semilinear (obstacle-like) problems of the form
\begin{equation}\label{eq:pde1}
\Delta u = f(u) , \quad \hbox{in } u > 0, 
\end{equation}
where $f$ is subject to certain conditions, listed below in (A1--4). 
In addition we expect $|\nabla u| =0$ on $\partial \{u > 0\}$, that characterizes the free boundary condition.
One variation of our problem is a simple perturbation of the obstacle problem, such as the equation
\begin{equation} \label{e:oblog}
\Delta u = (2- \log(u))\chi_{\{u>0\}}
\end{equation}
studied in   \cite{qs17}, \cite{FK},  and \cite{ks21}. A second example is variations of the Alt-Phillips problem \cite{ap85} given by equations like
\[
   \Delta u = u^{\gamma}(2-\log(u))\chi_{\{u>0\}}, \qquad (-1/3 <   \gamma < 1).
  \]  
Our analysis  encompasses many modifications of these problems, even allowing $f$ to fluctuate between the sub- and supercritical ranges for the Alt-Phillips problem such as in
\[
\Delta u = u^{\alpha \sin(\log(2 -\log u))}\chi_{\{u>0\}}. 
\]

For this kind of one-phase problem, obtaining optimal estimates on the growth of $u$ away from the free boundary $\{u > 0\}$ is not difficult, and we do so in Section \ref{sec:prelim}. Our main interest, however, is in the regularity of the free boundary itself. Here it is helpful to consider \eqref{e:oblog}, and observe that at any free boundary point, the blow-up limits of solutions in the natural scaling for the equation solve the \emph{classical obstacle problem} (see \cite{ks21}). As such, there are two kinds of free boundary points: ones where the blow-up limit is a regular free boundary point for the obstacle problem (i.e. flat points) and ones where it is a singular point. In this paper, our goal is to study the free boundary near flat points, and show that it is locally a smooth hypersurface. For an equation like \eqref{e:oblog}, this turns out to be quite challenging, in part due to how different the PDE is from the one satisfied by the blow-up limit (they do not even have the same natural scaling).
 
 Since the seminal paper \cite{c77}, many techniques have been developed to study both the initial free boundary regularity as well as higher regularity for obstacle-type problems. Semiconvexity estimates \cite{f18} are useful for obtaining Lipschitz regularity and then $C^{1,\beta}$ regularity of the free boundary. In our situation when $f$ is in the subcritical range of Alt-Phillips (or even a perturbation of the obstacle problem such as \eqref{e:oblog}), the solution $u$ will no longer be semiconvex. Methods to obtain directional monotonicity from the equation that $\partial_e u$ satisfies (see \cite{psu12}) are also unavailable due to the possible singular behavior of $f$. Since \cite{ac85} the boundary Harnack principle has been a powerful tool to move from Lipschitz to $C^{1,\beta}$ regularity of the free boundary. This technique has seen renewed interest in recent years by utilizing the boundary Harnack inequality for inhomogeneous equations \cite{as19,rt21,aks23,t23}. Unfortunately, the boundary Harnack inequality corresponding to our situation may not hold.
 
 Typically, techniques evolve to handle broader problems over time. However, since modern methods seem inapplicable here, we prove that flat free boundaries are Lipschitz graphs using a more rigid comparison technique based on the Kelvin transform from \cite{ws07}.

 Higher regularity of the free boundary for the obstacle problem was originally obtained with the Legendre transform \cite{kn77}. More recently, the higher order boundary Harnack principle \cite{ds15} has proven quite adaptable for obstacle-type problems such as the thin obstacle problem \cite{ds15a,jn17,rt21}. An alternative approach to obtaining the higher-order boundary Harnack principle arises from acquiring Schauder estimates for a highly degenerate elliptic equation. This approach is discussed in \cite{ssv21, z23, tts24}, where $C^1$ regularity of the boundary is assumed.
 We utilize this new technique to prove Cordes-Nirenberg estimates for the highly degenerate elliptic equation to move from Lipschitz to $C^{1,\beta}$ regularity of the free 
 boundary.

 In both  \cite{tts24,z23} 
 the Schauder theory gives $C^{\infty}$ and real-analyticity respectively. But again in our problem, the method as given in  \cite{tts24} 
 and \cite{z23} (in particular how the free boundary is flattened) fails to move beyond $C^{1,\beta}$ regularity, see the discussion in Section \ref{s:hodograph}. Instead, we more carefully flatten the free boundary by utilizing a modified partial hodograph transform. We could then prove Schauder estimates as in  \cite{tts24,z23}  
 on the nonlinear equation obtained in the hodograph variables, but opt instead to prove Caffarelli-Peral estimates. These estimates are more easily obtained, and when applied in the hodograph variables they show that a $C^{1,\beta}$ free boundary is $C^{\infty}$.

 We point out two recent results involving the specific case when $f(t)=t^{\gamma} \chi_{\{t>0\}}$. In the paper \cite{fk24} the authors study a vector-valued version for $0\leq \gamma<1$ and obtain $C^{\infty}$ regularity of the regular part of the free boundary. In \cite{rr24}, the authors consider the scalar valued case when $-1<\gamma<0$ and prove that if a portion of the free boundary is $C^{1,\alpha}$, then one obtains $C^{\infty}$ regularity of the same portion of the free boundary. The above article covers the situation $-1<\gamma\leq -1/3$, which we do not treat; however, as explained below we consider a wider class of right hand side functions $f$.

\subsection{Assumptions on nonlinearity}

Below are the assumptions we make regarding the nonlinearity $f(t)$. Note that since our focus lies on the local behavior of solutions near $\{u = 0\}$, assumptions pertaining to $t \rightarrow \infty$ are not relevant for questions regarding existence and boundedness. Define
\[
    F(t) = \int_0^t f(s) ds,
\]
which is guaranteed to exist under Assumptions (A1,2) 
below for all $f$ we consider. 

For $-1 < \gamma_1 < \gamma_2 < 1$ and $M > 0$, we say that $f \in \mathcal{F}=\mathcal{F}(M,\gamma_1,\gamma_2)$ if the following assumptions are satisfied:

\begin{enumerate}
    \item[(A1)]  $f$ is continuous and nonnegative on $(0, 1]$.
    \item[(A2)]  $f$ satisfies the bounds
        \begin{equation} \label{e:fassume2} 
             \frac{1}{M} t^{\gamma_2} \leq f(t) \leq Mt^{\gamma_1} \quad \text{ for } 0<t\leq 1.
        \end{equation}
    \item[(A3)]  the primitive $F$ satisfies
         \begin{equation} \label{e:fassume3} 
         1+\gamma_1 \leq t \frac{F'(t)}{F(t)} \leq 1+\gamma_2 \quad \text{ for } 0<t\leq 1.
        \end{equation}
    \item[(A4)]  $f$ satisfies the additional bound
        \begin{equation} \label{e:fassume4}
         f(at)> a^{(n+2)/(n-2)}f(t) \quad \text{ whenever } 0< a < t \leq 1. 
        \end{equation}
\end{enumerate}

For the higher regularity results (Lipschitz free boundaries are smooth), our approach also requires the additional assumption
\[
    \gamma_1 > -\frac{1}{3}.
\]
The case of $\gamma_1 \leq -\frac{1}{3}$ results in a linearized problem with a differently structured boundary condition, and we leave this case as an interesting open problem.

The following easier-to-verify assumption on the derivative of $f$ implies  (A2--4) 
and already accounts for most of the examples we have in mind: $f$ is continuously differentiable on $(0, 1]$ and
\[
     \gamma_1 \leq t \frac{f'(t)}{f(t)} \leq \gamma_2 \quad \text{ for } 0<t\leq 1.
\]
In particular, (4) 
follows from this if $\gamma_2 < \frac{n + 2}{n - 2}$, which is weaker than the standing assumption that $\gamma_2 < 1$.

To explain the nature of these assumptions, first recall that for homogeneous $f(t) = t^\gamma \chi_{\{u>0\}}$ (the Alt-Phillips problem, studied in \cite{ap85}), the range $\gamma \in (-1, 1)$ exhibits free boundary behavior. If $\gamma \geq 1$, the resulting equation instead satisfies a strong maximum principle and has no free boundary. The range $\gamma \leq -1$ has very different monotonicity properties and free boundary behavior, and we do not consider it here (it has been studied recently in \cite{DSS23}). Assumption (A2) 
asks that $f$ remains in the  free boundary  range of Alt-Phillips without demanding any homogeneity of $f$.

This turns out to be not scale-invariant under the natural (implicit) scaling of the problem. The purpose of assumption (A3)  
is to guarantee that solutions to the ODE $h'' = f(h)$, if rescaled via $h_r(t) = h(r t)/h(r)$, continue to satisfy (A2). 
This allows for $f$ to oscillate between different homogeneities, but only in a controlled fashion (on a scale comparable to $\log t$). The assumption (A4) 
will be used to ensure that solutions interact favorably with their Kelvin transforms, and in particular 
(A4) follows from  
a comparison-type property.

Some examples of $f(t)$ satisfying our assumptions would be $f(t)=\chi_{\{t>0\}}$ or $f(t)=(c+g(t)) \chi_{\{t>0\}}$ with $c+g(t)> 0$ and $g(t)\in C^{1}$ such as in the classical obstacle problem. We allow for singular and degenerate perturbations of the classical obstacle problem such as $f(t)=(2-\log(t))\chi_{\{t>0\}}$ or $f(t)=(1/(2-\log(t)))\chi_{\{t>0\}}$. We allow for Alt-Phillips equations $f(t)=t^{\gamma}\chi_{\{t>0\}}$ for $\gamma \in (-1, 1)$ as well as perturbations such as $f(t)=t^{\gamma}(2-\log t)^{\pm1}\chi_{\{t>0\}}$. We allow $f(t)$ to oscillate between the sub and supercritical regions such as $f(t)=t^{\alpha \sin(\log(2 -\log t))}\chi_{\{t>0\}}$ 
for numbers $|a|$ small enough, as well as other functions oscillating at this rate.
To allowing oscillations, our methods do not utilize a particular blow-up limit of $u$, and hence that of  $f$, since in the rescaling we only obtain a function in the same class $\mathcal{F}(M,\gamma_1,\gamma_2)$.

\subsection{Main results}

 The solutions of \eqref{eq:pde1} we consider are minimizers of the functional
\begin{equation} \label{e:func}
J(w):=\int_{\Omega} \frac{1}{2} |\nabla w|^2 +F(w) \,  dx,
\end{equation}
 over the class of functions 
 \[
 \{w \in W^{1, 2}(\Omega) : 0 \leq w \leq 1, w - v \in W^{1, 2}(\Omega) \}
 \] 
 for some fixed $v \in W^{1,2}(\Omega)$; we call these $\Comp$-minimizers. Any minimizer of $J$ over competitors in $W^{1,2}(\Omega)$ will be a $\Comp$-minimizer on a small enough neighborhood of a free boundary point. See Definition \ref{d:bddmin} for further discussion.

 Throughout the paper, we will refer to the unique positive solution $h$ (see Proposition \ref{p:odeprop}) of the ordinary differential equation
 \begin{equation} \label{e:h}
  \begin{cases}
   &h'' (t) =f(h(t) ) \qquad   \hbox{in }t > 0 , \\
   &h(0)=h'(0)=0 .
  \end{cases}
 \end{equation}
 In Section \ref{sec:prelim} we show that a minimizer $u$ of \eqref{e:func} will exhibit the same asymptotic growth as $h$ from a free boundary point, namely 
 \[
  c h(r) \leq  \sup_{B_r(x)} u \leq C h(r) , \qquad 
  \forall \ x\in \partial \{u > 0\}.
 \]
We then examine \emph{flat points} of the free boundary: ones where, after a possible rotation and translation, 
 \begin{equation}\label{eq:flatness}
  h(x\cdot e_n - \epsilon) \leq u(x) \leq h(x\cdot e_n +\epsilon).
 \end{equation}

 Our main results regarding the free boundary near these flat points are outlined below. In Section \ref{s:flat}, we utilize the Kelvin transform to demonstrate that flatness implies directional monotonicity and Lipschitz regularity of the free boundary.

 \begin{theorem} \label{t:kelvin}
  Let $u$ be a $\Comp$-minimizer of \eqref{e:func} in $B_1$, and fix $\delta>0$. There exists $\epsilon>0$ depending on $\delta,n,M,\gamma_1,\gamma_2$, such that if $u$ satisfies  \eqref{eq:flatness} 
  then $\partial_\nu u \geq 0$ in $B_{1/2}$ whenever $\nu \cdot e_n \geq \delta$. In particular the free boundary $\partial \{ u > 0\} \cap B_{1/2}$ is a Lipschitz graph.
 \end{theorem}

Our main result is to then obtain $C^{\infty}$ regularity of the free boundary. In Section \ref{s:consequencea} we collect estimates needed later for degenerate elliptic equations. 

In Section \ref{s:hodograph}, we employ a variant of the partial hodograph transform to flatten the free boundary. We then elaborate on how to use $W^{1,p}$ Caffarelli-Peral estimates to yield $C^{\infty}$ regularity of the free boundary and prove our main result. 


 \begin{theorem} \label{t:cinfty}
   Assume $\gamma_1 > - \frac{1}{3}$. Let $u$ be a $\Comp$-minimizer of \eqref{e:func} in $B_1$. There exists $\epsilon>0$ depending on $n,M,\gamma_1,\gamma_2$ such that if $u$ satisfies  \eqref{eq:flatness} 
  then $\partial\{u>0\}\cap B_{1/2}$ is $C^{\infty}$.
 \end{theorem}

 In Section \ref{s:peral}, we establish the Caffarelli-Peral estimates.

 \section{Preliminaries}\label{sec:prelim}

We list some additional consequences for $f$ and $F$ from our assumptions (A1--4). 
 \begin{proposition} \label{p:fandF}
  Let $f \in \mathcal{F}$, and let $F$ be the primitive of $f$. For $0<t,at<1$ we have 
  \[
  \begin{aligned}
     &a^{1+\gamma_2}F(t) \leq F(at) \leq a^{1+\gamma_1}F(t) \quad \text{ for } a\leq 1. \\
     &a^{1+\gamma_1}F(t) \leq F(at) \leq a^{1+\gamma_2}F(t) \quad \text{ for } a >1. 
  \end{aligned}   
  \]
 \end{proposition}

 \begin{proof}
 First let $a\leq 1$. By scaling and considering $F_r:=F(tr)/F(r)$ we may assume $t=1$ and $F(1)=1$. Consider $G(t):=F(t)-t^{1+\gamma_2}$. Then $G(1)=0$ and $G(0)\geq 0$. If $G$ obtains an interior minimum at $t_0$ with $G(t_0)\leq 0$, then 
 \[
  0=G'(t_0)=F'(t_0) -(1+\gamma_2) t_0^{\gamma_2},
  \]
 or 
 \[
 F'(t_0)=(1+\gamma_2)t_0^{\gamma_2} .
 \]
 Since $G(t_0)\leq 0$, then $F(t_0) \leq t_0^{1+\gamma_2}$, so that 
 \[
  (1+\gamma_2)\leq t_0 F'(t_0)/F(t_0).
 \]
 From assumption \eqref{e:fassume3} we must have equality above so that $G(t)\geq 0$, and this gives the first inequality. A similar argument shows 
 \[
 F(at)\leq a^{1+\gamma_1}F(t).
 \]
By reversing the roles of $t,at$ we obtain the inequalities for $a>1$. 
 \end{proof}

 Here, we compile properties of the unique positive (for $t>0$) solution $h$ to the ordinary differential equation \eqref{e:h}. While equation \eqref{e:h} doesn't admit a single solution, as $h(t)\equiv 0$ is one solution and $h(t-t_0)$ is another whenever $h$ is a solution and $t_0$ is positive, we prove here the existence of a unique solution that is positive for $t>0$.

 \begin{proposition} \label{p:odeprop}
  Assume that $f \in \mathcal{F}$. Then there exists a unique positive solution $h$ to \eqref{e:h}. Furthermore $h$ satisfies the following properties. There exists a constant $C_1$  such that 

  \begin{equation} \label{e:hprop}
    \begin{split}
    & C_1^{-1} t^{2/(1-\gamma_2)}  \leq h(t) \leq C_1 t^{2/(1-\gamma_1)} \quad \text{ for } \ 0<t<1 \\
    & \frac{2}{1-\gamma_1} \leq t \frac{h'(t)}{h(t)} \leq \frac{2}{1-\gamma_2} \quad \text{ for } \  0<t<1\\
    & \frac{1+\gamma_1}{1-\gamma_1} \leq t \frac{h''(t)}{h'(t)} \leq \frac{1+\gamma_2}{1-\gamma_2} \quad \text{ for } \  0<t<1\\
    & K^{C_1} t \leq h^{-1}(K h(t)) \leq K^{C_1^{-1}} t \qquad  \text{ for all }  0<K\leq 1.\\
    & K^{C_1^{-1}} t \leq h^{-1}(K h(t)) \leq K^{C_1} t \qquad  \text{ for all }  K > 1.
    \end{split}
  \end{equation}
 \end{proposition}

 \begin{proof} 
   For $\gamma<1$, we note that a solution $h_{\gamma}$ to 
   \[
   \begin{cases}
     & y'' = c y^{\gamma} \\
     & y(0)=y'(0)=0 
   \end{cases}
   \]
   is given by 
   \[
    h_{\gamma}(t)=\left[c\frac{(1-\gamma)^2}{4\gamma} \right]^{1/(1-\gamma)} t^{2/(1-\gamma)}. 
   \]

  Thus, $h_{\gamma_2} \leq h_{\gamma_1}$, and $h_{\gamma_2}$ is a subsolution to \eqref{e:h} whereas $h_{\gamma_1}$ is a supersolution. Thus, by the method of sub and supersolutions, there exists a positive solution $h$ to \eqref{e:h} over the interval $0<t\leq 1$. We now show that there is a maximal such positive solution $h_{\max}$. Indeed, consider the solution to 
  
   \[
    \begin{cases}
     & y'' =  h(y) \\
     & y(0)=\epsilon, \quad y'(0)=0. 
   \end{cases}
   \]
   which is bounded  below by $h$. Then as $\epsilon \to 0$ we obtain $h_{\max}$. Furthermore, any positive solution is bounded from below by $h_{\max}(t-\epsilon)$. Then $h=h_{\max}$ is a unique positive solution. We note that \eqref{e:h} can be rewritten as solving 
   \[
   \begin{cases}
    &h'(t)=\sqrt{2F(h(t))} \\
    &h(0)=0. 
    \end{cases}
   \]
   The existence and uniqueness for $t>1$ is immediate since $h$ is increasing and $\sqrt{2F}$ is $C^1$ when $F$ is positive.

   If we rescale with $h_r(t):=h(rt)/h(r)$, then 
   \[
   h_r'(t)=\sqrt{2F_r(h_r(t))},
   \]
   with $F_r(t):=(r^2/h^2(r))F(th(r))$. We remark that 
   \[
   F_r(1)=\frac{r^2}{2} \frac{[h'(r)]^2}{h^2(r)},
   \]
   and that $F_r$ also satisfies \eqref{e:fassume3}. Also, $h_r(1)=1$. From Proposition \ref{p:fandF} we have $F_r(t)\geq t^{1+\gamma_2} F_r(1)$. Consider now the function $g(t)=t^{2/(1-\gamma_2)}$, and suppose that 
   \[
   r\frac{h'(r)}{h(r)} =\sqrt{2F_r(1)} \geq  (1+ \epsilon)\frac{2}{1-\gamma_2}. 
   \]
   Then 
   \[
   g'(t) = \frac{2}{1-\gamma_2} t^{(1+\gamma_2)/(1-\gamma_2)}< \sqrt{2F_r(1)} t^{(1+\gamma_2)/(1-\gamma_2)} = \sqrt{2F_r(1)} g(t)^{(1+\gamma_2)/2} \leq \sqrt{2F_r(g(t))}. 
   \]
   Then $g(t)$ is a subsolution and lies below the solution $h_r$ in a neighborhood of $t=1$. However, $h_r(1)=1=g(1)$, and this gives a contradiction. Then 
   \begin{equation} \label{e:hbelow}
   r\frac{h'(r)}{h(r)}=\sqrt{2F_r(1)}\leq \frac{2}{1-\gamma_2}. 
   \end{equation}
   A similar argument with $g(t)=t^{2/(1-\gamma_1)}$ as a supersolution gives 
   \begin{equation} \label{e:habove}
   F_r(1)=r\frac{h'(r)}{h(r)}\geq \frac{2}{1-\gamma_1}. 
   \end{equation}

   Next, we consider $th''/h'$. We have 
   \[
    t\frac{h''}{h'}=t\frac{f(h)}{h'}
    =t\frac{f(h)}{h'}\frac{h'}{h'}\frac{h}{h}
    =\frac{hf(h)}{2F(h)} t\frac{h'}{h} . 
   \]
   Using \eqref{e:hbelow}, \eqref{e:habove}, and assumption \eqref{e:fassume3} we conclude the estimate for $th''/h$. 
   
   Finally, just as in the proof for $F$ in Proposition \ref{p:fandF}, we obtain for $h$ that 
   \[
   \begin{aligned}
   h(K^{C_1} t)\leq Kh(t)\leq h(K^{C_1^{-1}}t) \quad \text{ for } K\leq 1 \\
   h(K^{C_1^{-1}} t)\leq Kh(t)\leq h(K^{C_1}t) \quad \text{ for } K> 1. 
   \end{aligned}
   \]
By applying $h^{-1}$ to all expressions in both inequalities above, we obtain the final two inequalities in \eqref{e:hprop}.
 \end{proof}

  We state here a scaling result for $F$. 
 \begin{proposition} \label{p:fzoom}
  Assume $f \in \mathcal{F}(M,\gamma_1, \gamma_2)$. Then the rescaling 
  \[
  F_r(t):=\frac{r^2}{h^2(r)} F(t h(r))
  \]
  satisfies $f_r \in \mathcal{F}(M,\gamma_1, \gamma_2)$ for all $0<r<1$. 
 \end{proposition}

 \begin{proof}
  It is easy to check that $F_r(t)$ satisfies \eqref{e:fassume3}. Then from Proposition \ref{p:fandF} we have 
  \[
   F_r(1)t^{1+\gamma_2} \leq F_r(t) \leq F_r(1)t^{1+\gamma_1}. 
  \]
 Now $F_r(1)=rh'(r)/h(r)$, so from Proposition \ref{p:odeprop}
  we have 
  \[
   \frac{2}{1-\gamma_1}t^{1+\gamma_2} \leq F_r(t) \leq 
   \frac{2}{1-\gamma_2}t^{1+\gamma_1}. 
  \]
 Finally, since $f$ satisfies \eqref{e:fassume4}
 \[
 \frac{f_r(at)}{f_r(t)}=\frac{f(ath(r))}{f(th(r))}\geq
 a^{(n+2)/(n-2)},
 \]
  so that $f_r$ also satisfies \eqref{e:fassume4}. 
 \end{proof}

To avoid making global assumptions on $f$, we will work with the following weaker notion of minimizer:
\begin{definition}\label{d:bddmin} Let $\Comp = \{u \in W^{1, 2}(\Omega) : 0 \leq u \leq 1\}$. Given an $f \in \mathcal{F}$, we say that $u$ is a \emph{$\Comp$-minimizer} on $\Omega$ if $u \in \Comp$, and for any $v \in \Comp$ with $u - v \in W^{1, 2}_0(\Omega)$ we have $J(u) \leq J(v)$ (with $J$ as in \eqref{e:func}).
\end{definition}

 Occasionally, we will utilize a rescaling such as $u_r(x):=u(rx)/h(r)$. In those instances, the rescaled function is  now a minimizer among functions $w(x) h(r)\leq 1$. 
 
 Any minimizer of $J$ is a $\Comp$-minimizer on a neighborhood of a point in $\partial \{u > 0\}$ at which $u$ is continuous, and so our results will apply to it. We refer to the classic works \cite{GG83} and \cite{G03} on the existence and basic regularity theory for minimizers in general, which is not our primary focus here. However, we do note the following proposition, which will be useful for constructing comparison functions:

 \begin{proposition} \label{p:exitsandreg}
  Assume $f\in \mathcal{F}$, and fix a $\phi \in \Comp$. Then there exists a $\Comp$-minimizer $u$ with $u - \phi \in W^{1,2}_0$, and there exists a constant $\beta>0$ 
  such that any $\Comp$-minimizer will be in the class   $C^{1,\beta}(\Omega')$ for any $\W' \cc \W$, with norm bounded only in terms of $M, \g_1, \g_2, n, \W, \W'$.
 \end{proposition}

We provide only a sketch of the proof to show how to fit our assumptions into the standard framework of the calculus of variations. It is not difficult to see that minimizers of $J$ over  \emph{arbitrary} $W^{1, 2}$ functions need not be regular, or even exist, as we made no assumptions on $F$ when $t \notin (0, 1)$.

\begin{proof}
    Let $\tilde{F}$ be defined via
    \[
        \tilde{F}(t) = \begin{cases}
            F(0) & t < 0 \\
            F(t) & 0 \leq t \leq 1 \\
            F(1) & t > 1,
        \end{cases}
    \]
    and $\tilde{J}$ the corresponding functional as in \eqref{e:func}. Take any function $u \in \Comp$ with $u - \phi \in W^{1,2}_0(\Omega)$. Then we claim that $u$ is a $\Comp$-minimizer of $J$ if and only if it is a minimizer of $\tilde{J}$ (over arbitrary functions $v \geq 0$ in $W^{1,2}(\Omega)$ with $v - \phi \in W^{1, 2}_0(\Omega)$). Indeed, as $J = \tilde{J}$ on $\Comp$, it is clear for all such competitors in $\Comp$. On the other hand, if $u$ is a $\Comp$-minimizer of $J$ and $v$ is a competitor in $W^{1, 2}_0(\Omega)$, note that $w = \max\{\min\{1, v\}, 0\} \in \Comp$ is also a valid competitor and has $\tilde{F}(w) = \tilde{F}(v)$, so
    \[
        \tilde{J}(v) = \tilde{J}(w) = J(w) \geq J(u).
    \]

    It is clear from Assumption (A2)  
    that $\tilde{F} \in C^{0, \alpha}(\R)$, and it follows from the direct method that $\tilde{J}$ admits a minimizer $u$ with $u - \phi \in W^{1,2}_0(\Omega)$ (we omit the details here). Arguing as above, $u \in \Comp$: if not, $w = \max\{\min\{1, u\}, 0\}$ has $\tilde{J}(w) \geq \tilde{J}(u)$, which leads to $|\nabla u| = 0$ a.e. over $\{u < 0\} \cup \{u > 1\}$, a contradiction. 
    This proves the existence claim. 
    The $C^{1, \beta}$ regularity follows directly from \cite{GG83}, Theorem 3.1.
\end{proof}

 We now bound the growth away from the free boundary $\partial\{u>0\}$.

  \begin{theorem} \label{t:optimalreg}
    Let $u$ be a $\Comp$-minimizer of \eqref{e:func} in $B_1$ with $f \in \mathcal{F}$, and assume $x_0 \in B_{3/4} \cap \partial \{u>0\}$. There exists $C>0$ depending on $M,\gamma_1,\gamma_2$ and $\| u \|_{H^1(B_1)}$ such that 
    \[
     \sup_{B_r(x_0)} u \leq C h(r). 
    \]
  \end{theorem}

  \begin{proof}
   Suppose by way of contradiction the result is not true. Notice that by the $C^{1,\beta}$ regularity, we only need to prove the statement for $r$ small.

   Then there exist 
   \[
   \begin{cases}
    &u_k \quad\Comp\text{-minimizers of } \eqref{e:func} \\
    & x_k \in B_{3/4} \cap \partial \{u_k>0\} \\
    & r_k \to 0
   \end{cases}
   \]
   such that 
\begin{equation}\label{eq:sequence}
        S_{r_k}:=\sup_{B_{r_k}(x_k)} u_k \geq k h_k(r_k).
\end{equation}
   Furthermore, since the conclusion is true for $r>r_0$ we may assume additionally\footnote{This is achieved by selecting $r_k$ as the maximum radius for which \eqref{eq:sequence} is valid, ensuring that it also holds with equality.   }
   \[
    \frac{S_{2r_k}}{h_k(2r_k)} \leq \frac{S_{r_k}}{h_k(r_k)},  
   \]
   so that 
   \begin{equation} \label{e:sr/2}
    S_{2r_k} \leq \frac{h_k(2r_k)}{h_k(r_k)} S_{r_k} \leq C S_{r_k}. 
   \end{equation}
    We now rescale with 
    \[
    \tilde{u}_k := \frac{u_k (x_k + r_kx)}{S_{r_k}},
    \]
    noting that $ \sup_{\partial B_1} 
    \tilde u_k =1$, it is bounded on $B_2$ 
    and 
    that $\tilde{u}_k$ is a $\Comp$-minimizer on $B_{1/r_k}$ of \eqref{e:func}
    with 
    \[
    \tilde{F}_k(w) = \frac{r^2}{S_{r_k}^2} F_k (w S_{r_k}). 
    \]
   If $t_k$ is the unique point satisfying $h_k(t_k)=S_{r_k}\geq k h(r_k)$, then $r_k/t_k \to 0$. 
    From Proposition \ref{p:fzoom} we have that 
    \[
    \tilde{F}_k(w) = \frac{r_k^2}{t_k^2}\frac{t_k^2}{h_k^2(t_k)} F_k (w h_k(t_k)) \to 0 \text{ as }  k \to \infty. 
    \]
    In the limit we then have the following true in $B_2$
    \[
    \begin{cases}
    &u_k \rightharpoonup u_0 \text{ in } H^1 \\
    &u_k \to u_0 \text{ uniformly } \\
    &u_0 \geq 0  \\ 
    &\Delta u_0 =0. 
    \end{cases}
    \]
    Moreover, from uniform convergence $u_0(0)=0$. However, we also have in the limit that 
    \[
    \sup_{\partial B_1} u_0 =1, 
    \]
    but this will contradict the minimum principle. 
  \end{proof}

  We now aim to prove a nondegeneracy growth condition from a free boundary point. 

  \begin{lemma} \label{l:enondegen}
    Let $u$ be a $\Comp$-minimizer of \eqref{e:func} on $B_1$ and assume $f \in \mathcal{F}$. There exists $\epsilon=\epsilon(M,\gamma)$ such that if $u \leq \epsilon$ on $\partial B_1$, then 
    \[
     u(x) \equiv 0 \quad \text{ for } x \in B_{1/2}.
    \]
  \end{lemma}

  \begin{proof}

   Let $\overline{u}$ be a $\Comp$-minimizer for \eqref{e:func} with $\overline{F}(t)=\frac{M^{-1}}{\gamma_2+1}t^{\gamma_2+1}$ and with $\overline{u}=\epsilon$ on $\partial B_1$. We may assume $\gamma_2\geq 0$, so that $\gamma_2+1\geq 1$. As a consequence the corresponding functional $\overline{J}$ is convex, so the $\Comp$-minimizer $\overline{u}$ is unique. From symmetrization techniques, we have that $\overline{u}$ is radially symmetric and increasing as a function of the radius. We may apply the nondegeneracy estimate for supercritical Alt-Phillips; however one can show directly that $\overline{u}\equiv 0$ on $B_{1/2}$, or we may apply the nondegeneracy estimate from \cite{ap85}.

   We will now show that if $u$ is a minimizer of $J$, then $u \leq \overline{u}$. We note that since $\overline{F}' \leq F'$, then 
   \[
   F(u)-F(\overline{u})\geq \overline{F}(u)-\overline{F}(\overline{u})
   \]
   whenever $u \geq \overline{u}$, so that 
   \begin{equation} \label{e:compare}
    F(u)+\overline{F}(\overline{u})\geq \overline{F}(u)+F(\overline{u}) \quad \text{ whenever } u \geq \overline{u}.  
   \end{equation}
   Let us define $w_1 = \min\{u,\overline{u}\}$ and $w_2=\max\{u, \overline{u}\}$. Then 
   \[
   \begin{aligned}
    J(u)+\overline{J}(\overline{u}) 
    &=\int_{B_1} \frac{1}{2}|\nabla u|^2 + \frac{1}{2}|\nabla \overline{u}|^2 + F(u)+\overline{F}(\overline{u})\\
    &=\int_{\{u\geq \overline{u}\}} \frac{1}{2}|\nabla u|^2 + \frac{1}{2}|\nabla \overline{u}|^2 + F(u)+\overline{F}(\overline{u}) \\
    &\quad + \int_{\{u <\overline{u}\}} \frac{1}{2}|\nabla u|^2 + \frac{1}{2}|\nabla \overline{u}|^2 + F(u)+\overline{F}(\overline{u}) \\
    &=\int_{\{u\geq \overline{u}\}} \frac{1}{2}|\nabla w_2|^2 + \frac{1}{2}|\nabla w_1|^2 + F(u)+\overline{F}(\overline{u}) \\
    &\quad + \int_{\{u <\overline{u}\}} \frac{1}{2}|\nabla w_1|^2 + \frac{1}{2}|\nabla w_2|^2 + F(u)+\overline{F}(\overline{u}) \\
    &=\int_{\{u\geq \overline{u}\}} \frac{1}{2}|\nabla w_2|^2 + \frac{1}{2}|\nabla w_1|^2 + F(u)+\overline{F}(\overline{u}) \\
    &\quad + \int_{\{u <\overline{u}\}} \frac{1}{2}|\nabla w_1|^2 + \frac{1}{2}|\nabla w_2|^2 + F(w_1)+\overline{F}(w_2) \\
     &\geq \int_{\{u\geq \overline{u}\}} \frac{1}{2}|\nabla w_2|^2 + \frac{1}{2}|\nabla w_1|^2 + F(\overline{u})+\overline{F}(u) &\text{ from \eqref{e:compare}}\\
    &\quad + \int_{\{u <\overline{u}\}} \frac{1}{2}|\nabla w_1|^2 + \frac{1}{2}|\nabla w_2|^2 + F(w_1)+\overline{F}(w_2) \\
    &=J(w_1)+\overline{J}(w_2).
   \end{aligned}
   \]
   Thus, $w_1$ is a $\Comp$-minimizer of $J$ and $w_2$ is a $\Comp$-minimizer of $\overline{J}$. Since the minimizer of $\overline{J}$ is unique, we conclude that $w_2\equiv \overline{u}$.

   We will now show that $\overline{u}\equiv 0$ on $B_{1/2}$. 
   Now $\overline{u}/\epsilon$ is a $\Comp$-minimizer with $\epsilon^{\gamma_2-1}F$. We will use 
   $v(x)=2^{2/(1-\gamma_2)}(|x|-1/2)_+^{2/(1-\gamma_2)}$ as a barrier. 
   Note that 
   \[
   \Delta v(x) = 2^2[\beta(\beta+n-2)+|x|/(|x|-1/2)] v^{\gamma_2}.
   \]
   As long as $2^2[\beta(\beta+n-2)+|x|/(|x|-1/2)]\leq \epsilon^{\gamma_2-1}M^{-1}$, then the argument above (showing that $u\leq \overline{u}$) will also apply so that $\overline{u}\leq v \epsilon$. This shows that $\overline{u}\equiv 0$ on $B_{1/2}$, and therefore $u \equiv 0$ on $B_{1/2}$.  
  \end{proof}

  \begin{theorem} \label{t:nondegen} (Optimal non-degeneracy)
   Let $u$ be a $\Comp$-minimizer of \eqref{e:func} on $B_1$ and assume $f \in \mathcal{F}$. There exists $c$ depending on $M,\gamma_2$ such that if  $0 \in \partial \{u>0\}$, then 
   \[
   \sup_{B_r} u \geq c h(r) \quad \text{ for all } r\leq 1/2.  
   \]
  \end{theorem}

  \begin{proof}
    We choose $\epsilon$ from Lemma \ref{l:enondegen} for $f \in \mathcal{F}$ with $C$ from Proposition \ref{p:fzoom}. We then apply Lemma \ref{l:enondegen} to the rescaling 
    \[
    u_r(x):=\frac{u(rx)}{h(r)}
    \]
    to conclude that 
    $
     \sup_{\partial B_{r/2}}\geq \epsilon. 
    $
  \end{proof}

\section{Flat implies Lipschitz} \label{s:flat}
  We begin this section by giving the heuristic ideas behind combining the Kelvin transform and the lattice principle for functionals to conclude monotonicity in a cone of directions. This method was used in \cite{ws07} and will work for functionals of the form 
  \[
   J(w):=\int_{B_R} \frac{1}{2} |\nabla w|^2 +F(w),
  \]  
  as long as $F$ satisfies (A4) 
  and is nondecreasing.

  We use the Kelvin transformation on the ball $B_1(te_n)$ with $0<t\leq 3/2$ and denote 
  \begin{equation} \label{e:kelvinpoint}
   x^*=\frac{x-te_n}{|x-te_n|^2} + te_n.
  \end{equation}  
 We will let $0<h<t$, and let $R=2+1/(t-h)$. Suppose that we have a $\Comp$-minimizer $u$ in $B_R(0)$ which satisfies $u(x)\geq |x-te_n|^{2-n}u(x^*)$ whenever $x=(x',x_n) \in B_1(te_n) \cap \{x_n=h\}$. 
 From Assumption (A4) 
 we have $f(a^{n-2}t)> a^{n+2}f(t)$ when $0<a<1$, so that 
 \begin{equation} \label{e:increasing}
 a^{-2n}F(a^{n-2}t)-F(t)> a^{-2n}F(a^{n-2}s)-F(s),
 \end{equation}
 whenever $0<a< 1$ and $0<s< t\leq 1$. 
 If $\Omega = B_1(te_n) \cap \{x_n\leq h\}$, then $|x-te_n|^{2-n}u(x^*)\leq u(x)$ on $\partial \Omega$. We now show how the lattice principle will give the same inequality on all of $\Omega$. If $x \in \Omega$, we define
  $v(x)=|x-te_n|^{2-n}u(x^*)$ and note that $v$ is a minimizer  to 
  \[
  \tilde{J}(w):=\int_{\Omega} \frac{1}{2} |\nabla w|^2 +|x-te_n|^{-2n}F(|x-te_n|^{n-2}w), 
  \]
  among the class $w(x)|x-te_n|^{n-2}  \leq 1$. 
  We let $w_1 = \min \{v,u\}$ and $w_2 = \max\{v,u\}$. 
  Then 
  \[
  \begin{aligned}
   J(u)+\tilde{J}(v) &= J(u) + J(v) + 
   \int_{\Omega}(|x-te_n|^{-2n}F(|x-te_n|^{n-2}v)-F(v) \\
   &=\int_{\{ u \geq v\} } |\nabla u|^2 + F(u) + \int_{\{u < v\}} |\nabla u|^2 + F(u) \\
   &\quad +  \int_{\{u < v\}} |\nabla v|^2 + F(v) + \int_{\{u \geq v\}} |\nabla v|^2 + F(v) \\ 
   &\quad +  \int_{\Omega}(|x-te_n|^{-2n}F(|x-te_n|^{n-2}v)-F(v) \\
   &=J(w_2) + J(w_1) + \int_{\Omega}(|x-te_n|^{-2n}F(|x-te_n|^{n-2}v)-F(v)\\
   &\geq  J(w_2) + J(w_1) + \int_{\Omega}(|x-te_n|^{-2n}F(|x-te_n|^{n-2}w_1)-F(w_1)\\
   &=J(w_2) +\tilde{J}(w_1) .
   \end{aligned}
  \]
  where the penultimate inequality is coming from \eqref{e:increasing}, and the inequality is strict unless $v\equiv w_1$.  
  Since $u$ and $v$ are respective minimizers of $J$ and $\tilde{J}$ with appropriate boundary data, then we must have equality above,  so that $w_1\equiv v$ or $v \leq u $ in $\Omega$. We then have that 
  $u(x^*) \leq |x-te_n|^{2-n}u(x^*)\leq u(x)$ for $x \in \Omega$. 
  
  We now explain how this inequality can give directional monotonicity. Given a point $x \in B_{1/2}$, a direction $\nu \in S^{n-1}$ with $\nu \cdot e_n >0$, and $\tau$ chosen small, we consider $x+\tau \nu$. By a translation, we perform the Kelvin transform on the ball $B_1(te_n + y')$ for a $y$ chosen 
  such that $(x+\tau \nu)^*=x$. By choosing the appropriate value of $h$, then necessarily  $x \in \Omega$, so that from the argument above we conclude $u(x+\tau \nu) \geq u(x)$. This will give monotonicity of $u$ in the $\nu$ direction. To apply this method, one must consider a specific problem to determine which values of $t,h,\nu$ are admissible. In our specific free boundary problem, we will consider the free boundary trapped in the strip $\{|x_n| < \epsilon\}$ for small $\epsilon$ over a large ball $B_{R}$. By choosing $\epsilon$ small enough, we will be able to choose $t$ small which will allow us to reflect not only near the south pole (when $t$ is close to 1) which will give a small cone of monotonicity, but we will also be able to reflect near the equator (when $t$ is close to zero) and obtain a very large cone of monotonicity. As $\epsilon$ becomes smaller, the cone of monotonicity increases, and we are able to conclude $C^1$ regularity of the free boundary. The next two  lemmas carry out this argument. We begin by showing that in our situation $|x-te_n|^{2-n}u(x^*)\leq u(x)$ on $\partial B_1(t e_n) \cap \{x_n=h\}$.

\begin{lemma} \label{l:boundary}
    Fix $0<t_0<3/2$, and let $x^*$ be the reflected point via the Kelvin transform on the ball $B_1(te_n)$ as defined in \eqref{e:kelvinpoint}. There exists $l_0>0$ depending on $n,t_0,M,\gamma$ such that if $0\leq x_n\leq l_0$, and $t_0<t<3/2$, and if $x=(x',x_n)\in B_1(te_n)$, then 
    \begin{equation}  \label{e:reflect}
     |x-t e_n|^{2-n}h(x^*\cdot e_n) \leq h(x\cdot e_n). 
    \end{equation}
  \end{lemma}

  \begin{proof}
    We denote $\rho=|x'|$. When $\rho^2=\rho_1^2 :=1-(t-x_n)^2$, so that  $x=(\rho_1,h)\in \partial B_1(t e_n)$, the inequality \eqref{e:reflect} is trivially satisfied since $x=x^*$. Also, if $\rho^2\leq \rho_0^2=(t-x_n)/t-(t-x_n)^2$, then $u(x^*)=0$ and \eqref{e:reflect} is trivially satisfied. Thus, we only need to check when $\rho$ is in the interval $(\rho_0, \rho_1)$.     
    Now $x_n$ is fixed, and so we need 
    \[
     G(\rho):=h(x_n)|x-te_n|^{n-2}-h(x^* \cdot e_n)\geq 0. 
    \]
    Since $x_n$ is fixed, the term $|x-te_n|$ only depends on $|x'|=\rho$, so for notational convenience we define $g(\rho)=|x-te_n|$. 
    Suppose now that $G'(\rho)=0$ so that 
    \[
    0=G'(\rho)=h(x_n)(n-2)g^{n-3}(n-2)g'
    -h'(x^*\cdot e_n)2(t-x_n)g^{-3}g'.
    \]
    Using that $g'>0$ we have 
    \[
    h(x_n)g^n(n-2)=2(t-x_n)h'(x^* \cdot e_n). 
    \]
    If we now assume that at the same $\rho$, that $G(\rho)<0$,
    then 
    \[
    (n-2)g^2(\rho)h(x^*\cdot e_n)> h(x_n)g^n(\rho)(n-2)
    =2(t-x_n)h'(x^* \cdot e_n).
    \]
    Using \eqref{e:hprop} we have 
    \[
    (n-2)g^2(\rho)h(x^*\cdot e_n)\geq \frac{C(t_0-x_n)}{x^*\cdot e_n}h(x^*\cdot e_n).
    \]
    Taking $x_n$ small enough (so that $x^* \cdot e_n$ is arbitrarily small), we obtain a contradiction. 
  \end{proof}

\begin{lemma}   \label{l:qualitate}
   Fix $\eta>0$, and let $u$ be a $\Comp$-minimizer of \eqref{e:func} in $B_R(0)$ with 
   $R=2+2/\eta$. There exists $\epsilon_0>0$ depending on $\eta,n$ such that if 
   $h(x_n -\epsilon) \leq u \leq h(x_n + \epsilon)$ in $B_R(0)$, then $\partial_{\nu} u \geq 0$ for any $x \in B_1$ and $\nu \in \partial B_1$ with
   $e_n \cdot \nu \geq 2\eta$. 
  \end{lemma}
  
  \begin{proof}
   We first let $\epsilon < \eta/16$. By letting $\epsilon \to 0$, and by compactness and $C^{1,\beta}$ convergence, we have for small enough $\epsilon$ that $\partial_{\nu} u(x) \geq 0$ for $e_n \cdot \nu \geq 2\eta$ and $x \in B_{R/2}\cap \{x_n \geq \eta/8\}$. Since $u(x',x_n)=0$ for $x_n \leq -\epsilon$, we also have 
   $\partial_{\nu} u(x',x_n) \geq 0$ if $x_n \leq -\epsilon$. We now consider the more difficult strip where $\{|x_n|<\epsilon\}$, and we will employ the ideas from Lemma \ref{l:boundary}. For our ball $B_1(te_n)$ we will assume that $t \geq \eta$. This now gives a fixed $l_0$ coming from Lemma \ref{l:boundary}, and if necessary we will choose $\epsilon$ even smaller so that $\epsilon < l_0/4$. Since we used a first derivative argument in Lemma \ref{l:boundary} to prove \eqref{e:reflect}, we will utilize $C^{1,\beta}$ convergence (as $\epsilon \to 0$) for $x_n \geq l_0/2$. Then $u(x',x_n)\geq |(x',x_n)-t e_n|^{n-2}u(x^*)$ as long as $t \geq \eta/4$ and $x_n=l_0$.  Then from the argument explained at the beginning of this section, we have that $\partial_{\nu} u(x) \geq 0$ for $x \in B_1$ and $\nu \cdot e_n\geq 2\eta$.   
\end{proof}

From Lemma \ref{l:qualitate} and rescaling we obtain Theorem \ref{t:kelvin}. 

\section{Consequence of directional monotonicity} \label{s:consequencea}

In the previous section we showed that ``flat" implies directional monotonicity for $u$ as well as a Lipschitz free boundary. The next result will show that directional monotonicity will imply flatness. We illustrate this in  Corollary \ref{c:rescaledflat} which will give flatness at smaller scales. 
\begin{lemma}\label{l:fbflatimpliesflat}
	Let $u$ be a $\Comp$-minimizer of \eqref{e:func} on $B_1$ and assume $f \in \mathcal{F}(M, \g_1, \g_2)$ and $0 \in \partial \{u > 0\}$. For every $\d > 0$, there exists an $\e > 0$ (depending only on $M, \g_1, \g_2, n$) such that if  $u_e \geq 0$ for all $e\in S^{n-1}$ with $e \cdot e_n \geq \e$, then
	\[
		h(x_n - \d) \leq u(x) \leq h(x_n + \d)
	\]
	on $B_{1/2}$.
\end{lemma}

\begin{proof}
	We argue by contradiction. Assume not: then there is a sequence $F_k$ with $f_k \in \mathcal{F}(M, \g_1, \g_2)$, and $u_k$ corresponding $\Comp$-minimizers on $B_1$ with $(u_k)_e \geq 0$ for all $e \in S^{n-1}$ with $e \cdot e_n \geq \e_k \rightarrow 0$, but at some $x_k \in B_{1/2}$ we have
	\begin{equation}\label{e:compactnessalt}
		h_k((x_k)_n - \d) > u_k(x_k) \ \text{ or } \ h_k((x_k)_n + \d) < u_k(x_k),
	\end{equation}
	where $h_k$ is the solution associated to $f_k$ from Proposition \ref{p:odeprop}. From Assumption (A2), the $F_k(t)$ are uniformly H\"older continuous, and so we may extract a subsequence $F_k \rightarrow F$ locally uniformly. This $F$ will satisfy (A2-3), and by the same argument as in Proposition \ref{p:odeprop} (which does not require continuity from (A1) or (A4)) admits a unique positive solution $h$ to $h' = \sqrt{2 F(h)}$ with $h(0) = 0$; this is the unique minimizer to \eqref{e:func} in 1D with $0$ as a free boundary point. It is straightforward to verify that $h_k \rightarrow h$ in $C^1$ topology, by passing the ODE to the limit and using the estimates on $h_k$ to see that the limit is nonzero.

	From Proposition \ref{p:exitsandreg}, $u_k$ have $\|u_k\|_{C^{1, \a}(B_{3/4})}$ bounded uniformly in $k$, so we may find a subsequence $u_k \rightarrow u \in \Comp$ in $C^1$ topology. This $u$ will also be an $\Comp$-minimizer for $F$, and $u_e \geq 0$ for any $e$ with $e \cdot e_n > 0$. This immediately implies that $u_e = 0$ for $e$ with $e \cdot e_n = 0$, and so $u$ is a function of only $x_n$. We have that $u(0) = \lim u_k(0) = 0$, and also for any $r < 1/2$, from Theorem \ref{t:nondegen}, $\sup_{B_{r}} u_k \geq h(r) > 0$ uniformly in $k$. This implies that $ 0 \in \partial \{u > 0\}$ and in particular $u$ is nonzero.

	It follows that $u(x) = h(x_n)$ from the ODE uniqueness above. Assume that at $x_k$ we have the first alternative in \eqref{e:compactnessalt} along a subsequence with $x_k \rightarrow x$. Then from the fact that $h_k((x_k)_n - \d) > 0$ we must have $(x_k)_n  \geq \d$, and passing to the limit $h(x_n - \d) \geq h(x_n)$. As $h' = F(h) > 0$ is strictly increasing when positive, this is a contradiction. If the second alternative in \eqref{e:compactnessalt} occurs, we argue similarly after observing that as $(u_k)_e \geq 0$ and $u_k(0) = 0$, $u \equiv 0$ on the cone $\{x \in B_{1}: x_n < \e_k\}$. This means that as $u(x_k) > 0$, $(x_k)_n > - \e_k \rightarrow 0$ and $x_n \geq 0$. In the limit we get $h(x_n + \d) \leq h(x_n)$, leading to the same contradiction.
\end{proof}

\begin{corollary}\label{c:rescaledflat} Under the same assumptions as in Lemma \ref{l:fbflatimpliesflat}, we have that for every $r \leq 1$,
	\[
		h(x_n - r \d) \leq u(x) \leq h(x_n + r \d)
	\]
	for $x \in B_{r/2}$.
\end{corollary}

\begin{proof}
	Apply Lemma \ref{l:fbflatimpliesflat} to $u_r$, noting that it satisfies exactly the same assumptions as $u$, including the cone of monotonicity property.
\end{proof}

We now denote parametrize the Lipschitz free boundary by $(x', q(x'))$ with $q(x')$ a Lipschitz function. 

\begin{corollary} \label{c:4} Under the same assumptions as in Lemma \ref{l:fbflatimpliesflat}, we have that for $x \in B_{1/2}$,
	\[
	\begin{aligned}
		&(1) \quad (1 - \d) h(x_n) \leq u(x',x_n+q(x')) \leq (1 + \d) h(x_n) \\ 
		&(2) \quad \left|\frac{u_{x_n}(x',x_n+q(x'))}{h'(x_n)} - 1\right| \leq \d \\
		&(3) \quad \left|\frac{u_{x_i}(x',x_n+q(x'))}{h'(x_n)}\right| \leq \d \text{ for } i < n. \\
        &(4) \quad |D^2 u(x', x_n + q(x'))| \leq C \frac{h(x_n)}{x_n^2} \text{ and } u \in C^2(B_{1/2} \cap \{u > 0\})
    \end{aligned}
	\]
\end{corollary}

\begin{proof}
	We prove these properties with $x' = 0$ (and $q(x') = 0$), with the general case following by translating and rescaling. First, given a $\d_1$ to be chosen below in terms of $\d$, find $\e$ small such that Corollary \ref{c:rescaledflat} holds with $\d_1$.
	
	For (1), using Corollary \ref{c:rescaledflat} with $r = x_n$ gives
	\[
		h(x_n - x_n \d_1) \leq u(x) \leq h(x_n + x_n \d_1),
	\]
	and from the doubling properties in Proposition \ref{p:odeprop} this leads to
	\[
		h(x_n (1 + \d_1)) \leq h(x_n) (1 + \d_1)^{C_1} \leq (1 + C \d_1) h(x_n) \leq (1 + \d) h(x_n)
	\]
	for $\d_1$ small. A similar bound works from below.

	For (2) and (3), we use a simple interpolation argument. Set $v(x) = h(x_n)$; as both $u$ and $v$ are $\Comp$-minimizers, we have that
	\[
		\|u_r\|_{C^{1, \b}(B_{3/4})} + \|v_r\|_{C^{1, \b}(B_{3/4})} \leq C(M, \g_1, \g_2, n).
	\]
	for any $r < 1$. Focusing on $B = B_{T/2}(0, T) \ss B_{2T}$,
	\[
		[\nabla (u - v)]_{C^{0, \b}(B)} \leq C \frac{h(T)}{T^{1 + \b}}.
	\]
	From (1), we have (using that $|q(x')|\leq \e |x'|\leq 2 \e T$ and Proposition \ref{p:odeprop})
    that 
	\[
    \begin{aligned}
		|u(x) - v(x)| &=|u(x',x_n+q(x'))-v(x',x_n+q(x'))| \\
        &\leq \delta h(x_n) + |h(x_n + q(x')) - h(x_n)| \\
        &\leq \delta h(2T) + 2 \e T h'(2 T(1 + \e)) \\
        &\leq C\delta h(T)+ C h(2T(1+\epsilon))\\
        &\leq C h(T) [\d + \e].
        \end{aligned}
	\]
	on $B$.

	Then for any $\eta > 0$, there is a $C_\eta$ such that
	\[
		T \sup_{B} |\nabla u - \nabla v| \leq \eta T^{1 + \b} [\nabla (u - v)]_{C^{0, \b}(B)} + C_\eta \sup_B | u - v|
	\]
	(see Lemma 6.35 in \cite{GT83}, for example). Then
	\[
		\eta T^{1 + \b} [\nabla (u - v)]_{C^{0, \b}(B)} + C_\eta \sup_B | u - v| \leq C h(T) [\eta + C_\eta (\d_1 + \e)],
	\]
	so after choosing first $\eta$ small relative to $\d$, and then $\d_1, \e$ small in terms of $\eta$,
	\[
		\sup_{B} |\nabla u - \nabla v| \leq C [\eta + C_\eta (\d_1 + \e)] \frac{h(T)}{T} \leq \d h'(T).
	\]
	Plugging in $x = (0, T)$,
	\[
		|\nabla u(0, T) - h'(T) e_n| \leq \d h'(T).
	\]
	Then (2) and (3) follow from looking at the $e_n$ and $e_i$ components and dividing by $h'(T)$.

    For (4), consider the rescaled function $v = u_{T}$ on $B = B_{1/2}(0, 1)$. From (1), $|v|\leq C$ on $B$, while from (2,3), $|\nabla v| \leq C$ and $v_n \geq c$. Moreover, $v$ solves the semilinear equation $\Delta v = f_r(v)$ with $|f_r|\leq C$ on $B$. Write $\nu(x) = \frac{\nabla v(x)}{|\nabla v(x)|}$, which is continuous on $B$.  From Theorem 2.1 in \cite{N13},  it follows that each level set of $v$ is smooth, as is $\nabla v\cdot \nu(x)$ restricted to it, with derivatives uniformly bounded in terms of the bounds on $f_r$, $v$, and $\nabla v$. From Corollary 2.2 there, as $f_r$ is continuous, $D^2 v$ is continuous with $|D^2 v| \leq C$ (the latter is clear from the proof and Theorem 2.1). Scaling back to $u$, we see that $D^2 u$ is continuous and $|D^2 u(0, T)| \leq C \frac{h(T)}{T^2}$.
\end{proof}

This next result will be useful when we employ a modified Hodograph transform in the next section. 
 \begin{corollary} \label{c:linear}
   Under the same assumptions as in Lemma \ref{l:fbflatimpliesflat}, we have that for $x \in B_{1/2}$,
   if $u(x',z_n+q(x'))=h(x_n)$, then there exists a constant $C$ such that 
   \[
      \frac{1}{C}  \leq \left|\frac{z_n}{x_n} \right| \leq C .
   \]
 \end{corollary}

 \begin{proof}
  From $(1)$ in Corollary \ref{c:4} and \eqref{e:hprop} we have that 
  \[
  u(x',q(x')+(1-\zeta)x_n)\leq (1+\delta)h((1-\zeta)x_n)\leq -C_1 \ln(1-\zeta)(1+\delta)h(x_n)
  \]
  If $-C_1 \ln(1-\zeta)(1+\delta) < 1$, then since $h$ is increasing, we conclude that $z_n\geq(1-\zeta)x_n$. Then 
  \[
  \frac{z_n}{x_n} \geq (1-\zeta) > e^{-1/(C_1(1+\delta))}, 
  \]
  with $C_1$ the constant from \eqref{e:hprop}. A similar argument gives the bound from above.  
 \end{proof}

  We state here some weigheted Sobolev estimates that we will need in the following sections. We will consider the weighted measure $\alpha(x_n) \ dx$ with $\alpha(x_n) = (h'(x_n))^2$. Throughout the remainder of the paper, whenever we state that a constant depends on $\alpha$ we mean that the constant depends on $M,\gamma_1, \gamma_2$ where $f \in \mathcal{F}(M,\gamma_1, \gamma_2)$. After flattening the boundary, we will evenly reflect our functions, so that $u(x',x_n)=u(x',-x_n)$.
For smooth functions we consider the weighted norm
\begin{equation}\label{norm}
    \| u \|_{W^{1,p}(\alpha, \Omega)} := \left( \int_{\Omega} \alpha(x_n) (|u|^p+|D u|^p ) \right)^{1/p},
 \qquad \left( \alpha(x_n) = (h'(x_n))^2 \right),
\end{equation}
as well as 
\[
\| u \|_{L^p(\alpha, \Omega)}:= \left(\int_{\Omega} \alpha(x_n) |u|^p \right)^{1/p}. 
\]
Following the standard notation in the literature, we denote $H^{1,p}(\alpha,\Omega)$ as the closure of $C^{\infty}$ functions with the above norm. The space $W^{1,p}(\alpha,\Omega)$ typically refers to the set of $L^1_{\text{loc}}$ functions such that the function and weak derivative are bounded in the above norm. If $\gamma_2 >2$, then $W^{1,p}(\alpha,\Omega) \subsetneq H^{1,p}(\alpha,\Omega)$, so we will work in the space $H^{1,p}(\alpha,\Omega)$ which is complete.   The following inequality illustrates that if $\gamma_2 \leq 2$, the estimates in this section would be simpler. If $\gamma_2 \leq 2$, then $H^{1,2}(\alpha,\Omega)\subset L^1_{\text{loc}}(\Omega)$. We show later in Proposition \ref{p:sobolev2} that 
$H^{1,p}(\alpha,\Omega)\subset L^1_{\text{loc}}(\Omega)$ for $p$ large enough.

\begin{proposition} \label{p:sobolev1}
 Let $v \in H_0^{1,2}(\alpha,B_1)$, then for any $0<\gamma_3\leq \min\{1,2(1-\gamma_1)/(1+\gamma_1)\}$, there exists a constant $C$ depending on $\gamma_1,\gamma_2,\gamma_3,n$ such that\footnote{Observe that $\gamma_1$ and $\gamma_2$ appear indirectly through the definition of $\alpha(x_n)$ in \eqref{norm} and Proposition \ref{p:odeprop}. \label{footnote-depend}}
 \begin{equation} \label{e:wsobolev1}
  \int_{B_1} \frac{\alpha(x_n)}{|x_n|^{\gamma_3+1}} v^2 \leq C \int_{B_1} \alpha(x_n) |\nabla v|^2. 
 \end{equation}

\end{proposition}

\begin{proof}
 Let 
 $\phi \in C_0^1(B_1)$, and extend it by zero outside $B_1$.  For  fix $x' \in \mathbb{R}^{n-1}$ we have 
 \[
  0=\int_{-1}^1 \frac{d}{dt} \left[ \frac{\alpha(t)}{|t|^{\gamma_3}}\phi^2(x',t)\right] \ dt 
  = \int_{-1}^1 \left[\frac{\alpha'(t)}{|t|^{\gamma_3}}-\gamma_3\frac{\alpha(t)}{|t|^{\gamma_3+1}}\right]\phi^{2}(x',t) \ dt + \int_{-1}^1  \frac{\alpha(t)}{|t|^{\gamma_3}} 2 \phi(x',t) \phi_{x_n}(x',t)\ dt. 
 \]
 Using that $\alpha(t)=[h'(t)]^2$ and \eqref{e:hprop} we have 
 \[
 \frac{\alpha'(t)}{|t|^{\gamma_3}}-\gamma_3\frac{\alpha(t)}{|t|^{\gamma_3+1}} \geq c \frac{\alpha(t)}{|t|^{\gamma_3+1}}. 
 \]
 Then using Young's inequality we obtain 
 \[
  \int_{-1}^1 \frac{\alpha(t)}{|t|^{\gamma_3+1}} \phi^2(x',t)
  \leq C\int_{-1}^1 \alpha(t) \phi^2_{x_n}(x',t) \ dt 
  + \epsilon \int_{-1}^1 \frac{\alpha(t)}{|t|^{2\gamma_3}} \phi^2(x',t)\ dt. 
 \]
 Using that $\gamma_3\leq 1$ we can absorb the second term on the right hand side to conlcude \eqref{e:wsobolev1}. 
\end{proof}

We also have the following inequality for $p\geq 2$.
\begin{proposition} \label{p:sobolev2}
Let $v \in H_0^{1,p}(\alpha,B_R)$.  If $p > 2(1+\gamma_2)/(1-\gamma_2)$, then there exists $C$ depending on $\gamma_2, R$ such that 
\begin{equation} \label{e:equalsobolev}
 \int_{B_R} v^p \leq C \int_{B_R} |x_n|^{2(1+\gamma_2)/(1-\gamma_2)} |\nabla v|^p 
 \leq \int_{B_R} \alpha(x_n) |\nabla v|^p. 
\end{equation}
\end{proposition}

\begin{proof}
 Fix $x' \in \mathbb{R}^{n-1}$. By density we need only prove the result for $\phi \in C_0^1(B_R)$. We will integrate over $B_R^+$, and the result will be the same for $B_R^-$. We have 
 \[
 \begin{aligned}
  \phi^p(x',r)&=  \left( \int_r^R \phi_{x_n}(x',x) \ ds\right)^p \\
 &=\left( \int_r^R s^{-\tau}\phi_{x_n}(x',s)s^{\tau} \ ds\right)^p\\
 &\leq  \left(\int_r^R s^{-\tau p/(p-1)} \right)^{p-1}\left(\int_r^R s^{\tau p}\phi_{x_n}^p(x',s)  \right)\\
 &\leq C r^{p(1-\tau)} \left(\int_{r}^R s^{\tau p}\phi^p_{x_n}(x',s) \ ds \right)\\
  &\leq C r^{p(1-\tau)} \left(\int_{0}^R s^{\tau p}\phi^p_{x_n}(x',s) \ ds \right)\\
 \end{aligned}
 \]
 If $p>2(1+\gamma_2)/(1-\gamma_2)$, then we choose $\tau$ so that $p \tau =2(1+\gamma_2)/(1-\gamma_2)$. Then $\tau<1$, so that we may integrate in $r$ and then $x'$ to obtain \eqref{e:equalsobolev}.  
\end{proof}

We also have the following weighted Sobolev inequality. 
\begin{proposition} \label{p:sobolev3}
 Let $v \in H_0^{1,2}(\alpha,B_1)$. There exists $C$ and $p>2$ both depending on $n,\gamma_1,\gamma_2$ such that\footnote{Recall footnote \ref{footnote-depend}} 
 \[
  \| v \|_{L^p(\alpha,B_1)} \leq C\| \nabla v \|_{L^2(\alpha,B_1)}.
 \]
\end{proposition}

\begin{proof}
 By density we need only prove the result for $\phi \in C_0^1(B_1)$. We will adapt the usual proof of the Sobolev inequality in the subcritical case. For $y\in \Gamma:=\{y \in \partial B_2: e_n \cdot y>1/4\}$, we have $\phi(x+2y)=0$, so that

 \begin{equation} \label{e:sb1}
 -c(n)\phi(x)=-\int_{\Gamma}\phi(x) \ d\sigma(y) =\int_{\Gamma}\phi(x+2y)-\phi(x) \ d\sigma(y)= \int_{\Gamma} \int_0^2 \langle \nabla \phi(x+ty),y \rangle \ dt \ d\sigma(y). 
 \end{equation}
  Using H\"older's inequality on the right side on the product space $\Gamma \times (0,2)$, we have 
 \[
 -c(n)\phi(x) \leq \left(\int_{\Gamma}\int_0^2 \frac{1}{\alpha(x_n+ty_n)} \ dt \ d\sigma(y) \right)^{1/2}
  \left(\int_{\Gamma}\int_0^2 |\nabla\phi(x+ty)|^2 \alpha(x_n+ ty_n) \ dt \ d\sigma(y) \right)^{1/2}.
 \]
 We now square both sides of the above inequality and multiply both sides by $\alpha(x_n)^{1/q}$ to obtain 
 \[
  \alpha(x_n)^{1/q}\phi^2(x) \leq C\alpha(x_n)^{1/q}\int_{\Gamma}\int_0^2 \frac{1}{\alpha(x_n+ty_n)} \ dt \ d\sigma(y) 
  \int_{\Gamma}\int_0^2 |\nabla\phi(x+ty)|^2 \alpha(x_n+ ty_n) \ dt \ d\sigma(y).
 \] Now since $\alpha$ is an increasing function, we have that $\alpha(x_n)\leq \alpha(x_n + ty_n)$ for $y_n \geq 0$, so that 
 \[
 \begin{aligned}
  \alpha(x_n)^{1/q}\int_{\Gamma} \int_0^2 \frac{1}{\alpha(x_n+ty_n)} \ dt \ d\sigma(y)
  &\leq \int_{\Gamma} \int_0^2 \frac{\alpha(x_n+ty_n)^{1/q}}{\alpha(x_n+ty_n)}\ dt \ d \sigma(y) \\
  &=  \int_{\Gamma} \int_0^2 \alpha(x_n+ty_n)^{1/q-1}\ dt \ d \sigma(y) \\
  &\leq \int_{\Gamma} \int_0^2 \alpha(ty_n)^{1/q-1}\ dt \ d \sigma(y) \\
  &\leq C \int_{\Gamma} \int_0^2 t^{2(1-\gamma_2)/(1+\gamma_2)(1-1/q)} \ dt \ d \sigma(y) &\text{ by } \eqref{e:hprop}\\
  &\leq C(n,\gamma_2,q)
  \end{aligned}
 \]
 as long as $q>1$ is chosen close enough to $1$. Then continuing with the usual proof of the Sobolev inequality in the subcritical case, we have 
 \[
 \alpha(x_n)^{1/q} |u(x)|^2 \leq C ( \alpha|\nabla u|^2 * \chi_{B_2}|x|^{1-n}). 
 \]
 Applying Young's convolution theorem, we conclude
 \[
  \| \alpha(x_n)^{1/q} |u(x)|^2 \|_{L^q} \leq C \| \alpha(x_n) |\nabla u|^2 \|_{L^1} 
  \| \chi_{B_2} |x|^{1-n} \|_{L^q}. 
 \]
 By taking the square root of both sides we obtain the result with $p=2q$.

\end{proof}

In the Hilbert space setting, the following result is immediate. 
\begin{lemma} \label{l:wharmonexist}
 Assume that $\lambda I \leq A \leq \Lambda I$ and let $g^i \in L^2(\alpha,B_r)$ for $0\leq i \leq n$ and label $\bm{G}=(g^1, \ldots, g^n)$. Let $\psi \in H^{1,2}(\alpha,B_r)$, then there exists a unique $w$ with 
 $w - \psi \in H_0^{1,2}(\alpha,B_r)$ such that 
 \begin{equation} \label{e:wharmonic}
   \int_{B_r} \alpha(x_n) \langle  A \nabla w , \nabla v \rangle= \int_{B_r} \alpha(x_n) g^0 v - \langle  \bm{G}, \nabla w\rangle. 
 \end{equation}
 for all $v \in W_0^{1,2}(\alpha,B_r)$. 
\end{lemma}

Finally, we state the Caccioppoli inequality which is obtained with the usual proof. 
\begin{proposition} \label{p:cacc}
If $w$ is a subsolution in the sense that 
\[
\int_{B_1} \alpha(x_n) \langle A \nabla w, \nabla v\rangle \leq 0 \quad \text{ for all } 
v \in H_0^{1,2}(B_1) \text{ with } v\geq 0,
\]
then there exists a constant $C$ depending on dimension $n,\lambda, \Lambda$  such that 
for $0<r<R\leq 1$
\[
 \int_{B_r} \alpha(x_n) |\nabla w|^2 \leq \frac{C}{(R-r)^2} \int_{B_R}\alpha(x_n) w^2. 
\]
\end{proposition}

  This next estimate we will utilize later on a weighted solution. 
  \begin{lemma} \label{l:gradapprox}
  If $w$ is a solution to 
 \eqref{e:wharmonic} in $B_1$ with $g^i \equiv 0$ for all $i$ and $A \equiv I$, then there exists constants
 $C,\beta$ depending on $\alpha,n$ such that 
 \begin{equation} \label{e:firstn-1}
   \|\partial_{\nu} w \|_{C^{2,\beta}(B_{1/2})} \leq C \|w\|_{L^2(\alpha,B_1)},
 \end{equation}
 whenever $\nu \cdot e_n=0$. Furthermore, $w_{x_n}(x',0)=0$ and 
 \begin{equation} \label{e:n}
 \|w_{x_n} \|_{C^{0,1}(B_{1/2})} \leq C \|w\|_{L^2(\alpha,B_1)}.
 \end{equation}
\end{lemma}

\begin{proof}
 We first utilize the fact that our weight $\alpha$ is independent of the first $n-1$ variables, to show derivatives in the first $n-1$ variables are also solutions.  If $e$ is a unit vector orthogonal to $x_n$, and we consider the difference quotient 
  \[
  \phi_{e,\tau}(x',x_n):=\frac{\phi(x'+e,x_n)-\phi(x',x_n)}{\tau},
  \]
  for a smooth function $\phi$, then
  \[
  \phi_{e,\tau}(x',x_n)=\frac{1}{\tau}\int_0^\tau \langle \nabla \phi(x'+\tau,x_n),e\rangle, 
  \]
  so that if $r<1-h$, then  
  \[
  \begin{aligned}
   \int_{B_r}\alpha(x_n)|\phi_{e,\tau}(x',x_n)|^2 
   &\leq \int_{B_r}\alpha(x_n)\left(\frac{1}{\tau} \int_0^\tau |\nabla \phi(x'+t,x_n)| \ dt \right)^2\\
   &\leq \frac{1}{\tau} \int_0^\tau \int_{B_1}\alpha(x_n) |\nabla \phi(x',x_n)|^2  \ dt \\
   &= \int_{B_1} \alpha(x_n)|\nabla \phi(x',x_n)|^2. 
   \end{aligned}
  \]
  Thus, by density of $C^{\infty}$ in $H^{1,2}(\alpha,B_1)$ we have that 
  for small enough $\tau$, 
  \begin{equation} \label{e:differ1}
   \int_{B_r} \alpha(x_n) w^2_{e,\tau} \leq \int_{B_1} \alpha(x_n) |\nabla w|^2. 
  \end{equation}
  By linearity, $w_{e,\tau}$ is a solution to \eqref{e:wharmonic}. Applying the Caccioppoli inequality we have that 
  \[
  \int_{B_{r_1}} \alpha(x_n) |\nabla w_{e,\tau}|^2 
  \leq C \int_{B_{r_2}} \alpha(x_n) | w_{e,\tau}|^2
  \leq C \int_{B_{r_3}} \alpha(x_n) |\nabla w|^2
  \leq C \int_{B_{r_4}} \alpha(x_n) w^2.
  \]
  Thus, $w_e \in H^{1,2}(\alpha,B_{r_1})$. By iterating this procedure we obtain $D^{\boldsymbol{\beta}}u \in H^{1,2}(\alpha,B_{1/2})$ as long as the multi-index $\boldsymbol{\beta}$ is only in the first $n-1$ coordinates. By utilizing the Sobolev inequality in Proposition \ref{p:sobolev3} and the Caccioppoli inequality in Proposition \ref{p:cacc}, the standard Moser iteration technique for subsolutions gives
  \begin{equation} \label{e:moser}
   \lim_{p \to \infty}\left(\int_{B_{3/4}} \alpha(x_n) (D^{\boldsymbol{\beta}}w)^p \right)^{1/p}
   \leq C\int_{B_{7/8}} \alpha(x_n) (D^{\boldsymbol{\beta}} w)^2. 
  \end{equation}
  The standard proof shows that this implies 
  \[
   \| D^{\boldsymbol{\beta}}w \|_{L^{\infty}(B_{3/4})} \leq 
   C\int_{B_{7/8}} \alpha(x_n) (D^{\boldsymbol{\beta}} w)^2.
  \]
  Since $D^{\boldsymbol{\beta}} w$ is bounded, we may repeatedly apply the \textit{unweighted} 
  Sobolev embedding theorem finitely many times on each $n-1$ dimensional slice, so that  for $|x_n|\leq 1/4$
  \[
   \|w(\cdot,x_n) \|_{C^{2,\beta}(B'_{1/4})} \leq  \|w\|_{L^{\infty}(B_{1/2})}. 
  \]

 To obtain regularity of $D^{\boldsymbol{\beta}}w$ as $x_n$ changes as well as $w_{x_n}$, we now consider regularity in the $x_n$ direction. From \eqref{e:firstn-1} we have that $(\alpha(x_n) u_{x_n})_{x_n}=\alpha(x_n)\Delta_{x'} \in C^{0,\beta}(B'(\cdot,x_n))$. Since $w$ is even and bounded (from \eqref{e:moser}), the weak formulation \eqref{e:wharmonic} implies that $\lim_{x_n \to 0} \alpha(x_n) w(x',x_n)=0$ for  $H^{n-1}$ almost every $x'$. 
 Then for $x' \in B'_{1/4}$, we have the following computation. 
 \[
 \begin{aligned}
  \alpha(x_n)w_{x_n}(x',x_n)&= \alpha(x_n)w_{x_n}(x',x_n)- 0
   = \int_0^{x_n} \partial_{x_n}(\alpha(y)w_{x_n}(x',y)) \ dy \\
   &= \int_0^{x_n} \alpha(y) \Delta_{x'} w(x',y) \ dy 
   \leq C \int_0^{x_n} \alpha(y) \| w \|_{L^{\infty}(B_{1/4})} .
  \end{aligned}
 \]
 Now 
 \[
 \frac{d}{dt} \alpha(t) = \frac{d}{dt} (h'(t))^2 = 2 h'(t)h''(t)=2\sqrt{2F(h(t))}f(h(t))\geq 0.
 \]
 Then we conclude 
 \[
 \begin{aligned}
 |\alpha(x_n)w_{x_n}(x',x_n)| &\leq \int_0^{x_n} \alpha(y) |\Delta_{x'} w(x',y)|\ dy
 \leq C \|w\|_{L^{\infty}(B_{1/2})} \int_0^{x_n} \alpha(y) 
 \ dy 
 \\ &\leq C \|w\|_{L^{\infty}(B_{1/2})} \int_0^{x_n} \alpha(x_n)  \ dy
 = C \|w\|_{L^2(\alpha,B_1)} x_n \alpha(x_n),
 \end{aligned}
 \]
 so that 
 \[
 |w_{x_n}(x',x_n)|\leq Cx_n \|w \|_{L^2(\alpha,B_1)}. 
 \]
 This proves the Lipschitz estimate of $w_{x_n}$ on the thin space; the estimate off the thin space follows from interior regularity and scaling, and we have \eqref{e:n}. Finally, having shown $w_{x_n}$ is Lipschitz, (so in particular $D^{\boldsymbol{\beta}}w_{x_n}$ is Lipschitz), we conclude \eqref{e:firstn-1}.
\end{proof}

\section{Hodograph Transform} \label{s:hodograph}

From the results of the previous section, we consider a $\Comp$-minimizer $u$ of \eqref{e:func} on $B_r$ that is monotone in the $e_n$ direction and has $\partial \{u > 0\}$ a $C^{1, \alpha}$ hypersurface containing $0$ and with inward unit normal $e_n$ at $0$. From $(2)$ in Lemma \ref{l:fbflatimpliesflat}, we have that $u_{x_n}>0$ in $B_s \cap \{u>0\}$ for small enough $s$. To obtain higher regularity we perform a modified hodograph transform with respect to $h$ by defining, for $y = (y', y_n) \in B_s^+ = \{(y', y_n) \in B_s : y_n > 0 \}$, the quantity $v(y)$ to be the unique number such that
\begin{equation} \label{e:hodograph}
 u(y',v(y',y_n))=h(y_n). 
\end{equation}

The mapping $y \mapsto (y', v(y', y_n))$ is a homeomorphism from $B_s^+$ to $U \cap \{u > 0\}$, where $U$ a neighborhood of $0$, and extends continuously up to the boundary, mapping $\{y_n = 0\}$ to the free boundary $\partial \{u > 0\}$. 
We note that the map $(y',y_n)\mapsto (y',v(y',y_n))$ is the inverse for $\Phi(x',x_n)=(x',h^{-1}(u(x',x_n)))$. Now $u \in C^2$ away from the free boundary by Corollary \ref{c:linear}. Then both $h^{-1},u \in C^2$. Also, $D \Phi$ is lower triangular, and since both $(h^{-1})'>0$ and $u_{x_n} >0$, then $D \Phi$ is invertible. Then since both $h^{-1},u \in C^2$ away from the free boundary, the inverse function theorem implies that $v \in C^2$ as long as $y_n>0$.  Differentiating repeatedly we obtain the following relations:
\begin{align*}
	&u_n v_n = h'(y_n) \\
	&u_i + u_n v_i = 0 \\
	&u_{nn} v_n^2 + u_n v_{nn} = h''(y_n) \\
	&u_{in} v_n + u_{nn} v_i v_n + u_n v_{in} = 0 \\
	&u_{ii} + 2 u_{ni} v_i + u_{nn} v_i^2 + u_n v_{ii} = 0.
\end{align*}
Here and below, $u$ and its derivatives are always evaluated at $(y', v(y))$, $v$ and its derivatives are evaluated at $y \in B_s$, and $i$ is any number $1, \ldots, n-1$ corresponding to a tangential direction. Summing the fifth equation over $i$, adding a multiple of the third equation, and making substitutions for the mixed and lower-order derivatives of $u$, we obtain
\[
	\Delta u = [1 + \sum_i v_i^2] u_{nn} + \sum_i [2 v_{in} \frac{h' v_i}{v_n^2} - v_{ii} \frac{h'}{v_n}] = f(h).
\]
It is helpful to multiply this whole expression by $h'$, and then eliminate the $u_{nn}$ factor. After combining terms, this leads to the following second-order PDE for $v$:
\begin{equation} \label{e:nondivhodograph}
    \left[1 + \sum_i v_i^2\right] \left[\frac{h''h'}{v_n^2} - v_{nn}\frac{(h')^2}{v_n^3} \right] + \sum_i \left[2 v_{in} \frac{(h')^2 v_i}{v_n^2} - v_{ii} \frac{(h')^2}{v_n}\right] = h' f(h).
\end{equation}

We claim this can be rewritten in divergence form. To see this, first observe that
\[
	\frac{h''h'}{v_n^2} - v_{nn}\frac{(h')^2}{v_n^3} = [\frac{(h')^2}{2 v_n^2}]_n,
\]
and so
\[
	[1 + \sum_i v_i^2][ \frac{h''h'}{v_n^2} - v_{nn}\frac{(h')^2}{v_n^3}] = [(1 + \sum_i v_i^2)\frac{(h')^2}{2 v_n^2}]_n - \sum_i v_{in} \frac{(h')^2 v_i}{v_n^2}.
\]
Substituting into \eqref{e:nondivhodograph}, we get
\[
	[(1 + \sum_i v_i^2)\frac{(h')^2}{2 v_n^2}]_n + \sum_i [v_{in} \frac{(h')^2 v_i}{v_n^2} - v_{ii} \frac{(h')^2}{v_n}] = h' f(h).
\]
The remaining term can be rewritten as a tangential derivative:
\[
	v_{in} \frac{(h')^2 v_i}{v_n^2} - v_{ii} \frac{(h')^2}{v_n} = - [\frac{(h')^2 v_i}{v_n}]_i,
\]
as $h' = h'(y_n)$ has $h'_i = 0$. This leads to
\begin{equation} \label{e:divhodograph}
 \left[(1 + \sum_i v_i^2)\frac{(h')^2}{2 v_n^2}\right]_n - \sum_i \left[\frac{(h')^2 v_i}{v_n}\right]_i = h' f(h).
\end{equation}
We rewrite this as
\[
	\sum_k \partial_k \bm{H}^k(y_n, \nabla v) = q(y_n),
\]
where a key point is that $q(y_n) = h' f(h)$ is independent of $y_i$, for $i \neq n$.

Set $w = v_i$: this satisfies (in the weak sense, using that $v \in C^2$) the PDE
\[
	\sum_{k, j} \partial_k [\bm{H}_j^k(y_n, \nabla v) w_j] = 0,
\]
as $q_i = 0$ for $i \neq n$. The matrix $\bm{H}_j^k$ can be computed directly, and if $v_i$ is replaced with $p_i$ so that $\bm{H}_j^k$ has a generic input, then 
\[
	-\bm{H}_j^k(t, p) = (h')^2 \begin{bmatrix}
		\frac{1}{p_n} & 0 & \cdots & 0 & -\frac{p_1}{p_n^2} \\
		0 & \frac{1}{p_n} & \cdots & 0 & -\frac{p_2}{p_n^2} \\
		\vdots & \vdots & &\vdots & \vdots \\
		0 & 0 & \cdots & \frac{1}{p_n} &-\frac{p_{n-1}}{p_n^2} \\
		-\frac{p_1}{p_n^2} & -\frac{p_2}{p_n^2} & \cdots & -\frac{p_{n-1}}{p_n^2} & \frac{1}{p_n^3}(1 + \sum_i p_i^2).
	\end{bmatrix},
\]
Observe that this is symmetric and of the form $(h')^2 A(\nabla v)$.

We now show that $A(\nabla v)$ is uniformly elliptic and close to the Laplacian in $L^{\infty}$ norm. 
\begin{proposition} \label{p:hodographintPDE} Let $v : \bar{B}_s^+ \rightarrow \R$ be as defined above, and $w = v_i$. Then given $\delta>0$, there exists an $s > 0$ small such that for any $\phi \in C_c^\infty(B_s^+)$,
	\[
		\int (h'(y_n))^2 A(\n v)\nabla w \cdot \nabla \phi = 0,
	\]
	and moreover $\| A(\nabla v) - I \| \leq \delta.$ 
\end{proposition}

\begin{proof}
In the discussion above we already established that $w$ is a solution to the equation, so we need only show that $A(\nabla v)$ is uniformly elliptic and $\|A(\nabla v)-I\|< \delta$. 
We recall that $v_n=h'(y_n)/u_n$. We let $z_n$ be such that $u(y',z_n + q(y'))=h(y_n)$ where $q$ is the parametrization of the free boundary as given in Corollary \ref{c:4}. Then from $(1)$ in Corollary \ref{c:4} we have that 
\[
(1-\delta) h(z_n)\leq h(y_n) \leq (1+\delta)h(z_n). 
\]
From the last two inequalities in \eqref{e:hprop}, we have that if $Ky_n=z_n$, then
\[
|K-1|\leq \omega_1(\delta)
\]
with $\omega_1$ a modulus of continuity. Now the third inequality in \eqref{e:hprop} can be utilized the same way as $(A3)$ in Proposition \ref{p:fandF} to conclude that 
\[
 \left|\frac{h'(y_n)}{h'(z_n)}-1 \right|\leq \omega_2(\delta)
\]
with $\omega_2$ a modulus of continuity. We may then choose $\epsilon$ smaller in Corollary \ref{c:4} and simply write 
\[
\left|\frac{h'(y_n)}{h'(z_n)} -1\right|\leq \delta. 
\]
Then 
\[
\begin{aligned}
|v_n^{-1}(y',v(y',y_n)) - 1|
&=|v_n^{-1}(y',z_n + q(y')) - 1|\\
&=\left|\frac{u_n(z_n+q(y'))}{h'(y_n)}-1 \right|\\
&=\left|\frac{u_n(z_n+q(y'))}{h'(z_n)}\frac{h'(z_n)}{h'(y_n)}-1 \right|\\
&\leq \left|\frac{u_n(z_n+q(y'))}{h'(z_n)} -1\right|
\left|\frac{h'(z_n)}{h'(y_n)} \right|+ \left|\frac{h'(z_n)}{h'(y_n)}-1 \right|\\
&\leq C\delta. 
\end{aligned}
\]
Similarly, since $-v_i=u_i/u_n$ we have 
\[
\begin{aligned}
 |v_i(y',y_n)|&=\left|\frac{u_i(y',y_n)}{u_n(y',y_n)} \right|\\
 &=\left|\frac{u_i(y',z_n+q(y'))}{u_n(y',z_n+q(y'))} \right| \\
 &= \left|\frac{u_i(y',z_n+q(y'))}{h'(z_n)} \right|
 \left|\frac{h'(z_n)}{u_n(y',z_n+q(y'))} \right|\\
 &\leq \delta/(1-\delta), 
\end{aligned}
\]
with the last inequality coming from $(2)$ and $(3)$ in Corollary \ref{c:4}. 
  We then have that $|v_i|<C\delta$. It follows that $\| A(\nabla v)- I ||\leq C \delta$. By choosing $\delta$ smaller, we may write this as   $\| A(\nabla v)- I ||\leq \delta$
\end{proof}

To utilize this, we also need a boundary condition on $w$, that surprisingly is of a simple Neumann type. Extend the function $v$ and $w$ to all of $B_s$ through even reflection, such that $v(y', -y_n) = v(y', y_n)$. When the matrix $A$ is reflected, the coordinates $a_{in}$ and $a_{ni}$ for $i < n$ will no longer be continuous. However, the condition $\| A - I \|\leq \delta$ will still hold, which will be utilized in the estimates in Section \ref{s:peral}.

\begin{proposition} \label{p:hodographintPDE2} 
Proposition \ref{p:hodographintPDE} is valid for any $\phi \in C_c^\infty(B_s)$ as well, under the same assumptions.
\end{proposition}

\begin{proof}
	From above, we have that $C^{-1} \leq v_n \leq C$ and $|v_i|\leq C\d$ on $B_s^+$, so $v$ is a bilipchitz map from $\bar{B}_s^+$ onto its image, and $w$ is a $C^1$ function on $B_s^+$. Take any $\phi \in C_c^\infty(B_s)$, and consider
 
	\begin{equation}\label{eq:est}
		\int_{B_s^+} (h'(y_n))^2 A \nabla w \cdot \nabla \phi = \lim_{t\searrow 0} \int_{B_s^+ \cap {y_n > t}} (h'(y_n))^2 A \nabla w \cdot \nabla \phi = - \lim_{t\searrow 0}  \int_{B_s^+ \cap \{y_n = t\}} (h'(t))^2 \phi A \nabla w \cdot e_n,
		\end{equation}
	where the last expression is obtained by using   integration by parts and Proposition \ref{p:hodographintPDE}. Now  using the expressions for the derivatives $v_{in}$ and $v_{ii}$ and inequalities $(2)$ and $(4)$ from  Corollary \ref{c:4}  we have 
	\[
		|\nabla w(y)| \leq C \frac{|D^2 u(y', v(y))|}{u_n(y', v(y))} \leq C\frac{h(y_n)}{y_n^2 h'(y_n)}\leq  \frac{C}{y_n},
	\]
 which gives the following estimate for the integrand in  \eqref{eq:est}
	
	\[
		|(h'(t))^2 \phi A \nabla w\cdot e_n| \leq C (h'(t))^2 |\nabla w(y', t)| \leq \frac{C}{t^3} h^2(t) \leq C \frac{t^{2/(1 + \gamma)}}{t^3} = C t^{-\frac{1 + 3 \gamma}{1 + \gamma}}.
	\]
	For  $\gamma < - \frac{1}{3}$ this expression  tends  to $0$ uniformly, implying 
	\[
		\int_{B_s^+} (h'(y_n))^2 A \nabla w \cdot \nabla \phi = 0.
	\]
	The integral on the other half-ball similarly vanishes.
\end{proof}

We are now in a position to prove the following regularity theorem 
from which Theorem \ref{t:cinfty} follows:

\begin{theorem} \label{t:bootstrap} Let $v$ be as in Proposition \ref{p:hodographintPDE}. Then for each $k \in \N$, there is an $s = s(k) > 0$ such that $v(\cdot, x_n) \in C^{k}(B'_{s})$ for each $x_n$, with
	\[
		\sup_{B_{s}}\sum_{|\alpha|\leq k} |\partial^\alpha \nabla v| \leq C(k),
	\]
	where the sum is over all multiindeces $\alpha$ in the first $n-1$ variables only, i.e. $\alpha \in \Z_{\geq 0}^{n-1} \times \{0\}$. In particular, the level set $\p \{ u > 0\}$ is $C^\infty$ on a neighborhood of $0$.
\end{theorem}

\begin{proof}
	We will prove, by induction on $k$, the following estimate: for all $2 < p < \infty$, there is a constant $C_{p , k}$ and a number $s = s(p, k)$ such that
	\[
		\sum_{|\alpha| = k}\|\partial^\alpha \nabla v\|_{L^p(B_s)} \leq C_{p , k}.
	\]
	The conclusion then follows from Sobolev embeddings and the observation that $v|_{y_n = 0}$ is precisely $\partial \{u > 0\}$ parametrized as a graph over $\R^{n-1}$.

    To show the base case $k = 0$, we have from Proposition \ref{p:hodographintPDE2} that $w = v_i$ (with $1\leq i <n$) satisfies
	\[
		\text{div} \ \left[(h'(y_n))^2 A(\n v) \nabla w \right]= 0
	\]
	in the weak sense. As mentioned earlier $A(\nabla v)$ is not (necessarily) continuous; however, from Proposition \ref{p:hodographintPDE} we have $\| A(\nabla v)-I\|\leq \delta$. Therefore, we may apply Theorem \ref{t:peral} to conclude that $v_i=w \in H^{1,p}(\alpha,B_s)$ for every $p<\infty$. This establishes the base case $k=0$.

    To illustrate the method, we show in detail the case $k=1$ before proceeding with the general induction for $k\geq 2$. We differentiate the above equation in a direction $i \leq n-1$. Since $i \leq n-1$, the discontinuity of the coefficients $a_{in}$ across $x_n=0$ does not pose an issue. We thus obtain 
    \[
    \text{div} \ \left[(h'(y_n))^2 A(\n v) \nabla w_i \right]=
    -\text{div} \ \left[(h'(y_n))^2 \sum_{j\leq n} D^j A(\n v)v_{ij} \nabla w \right].
    \]
    Now $p \mapsto A(p)$ is analytic, and so $D^j (\nabla v)$ is bounded over $B_s$. Furthermore,  $\nabla w, v_{ij} \in L^p(\alpha,B_s)$ for every $p<\infty$ from the base case $k=0$, so we may once again apply Theorem \ref{t:peral} and conclude that $\nabla w_i \in H^1(\alpha, p)$ for every $p<\infty$. Notice that here (and in every inductive step), only one derivative in the $n$ direction appears.

    We now proceed with the induction.  Consider some $k \geq 2$. Fix $i \leq n - 1$, and set $w = v_i$, which we know satisfies
	\[
		\text{div} \ \left[(h'(y_n))^2 A(\n v) \nabla w \right]= 0
	\]
	in the weak sense. Now take a multiindex $|\alpha| = k - 1$ containing only tangential derivatives, and differentiate the equation:
 
	\[
		\text{div} \ \left[(h'(y_n))^2 A(\n v) \nabla \partial^\alpha w \right]= -\text{div} \ \left[(h'(y_n))^2 \sum_{\beta_1 + \beta_2 = \alpha, |\beta_1| > 0} \partial^{\beta_1}[A(\n v)] \partial^{\beta_2}w \right].
	\]
	The quantity $\partial^{\beta_1} [A(\n v)]$ can be expressed (if $\partial^{\beta_1} = \partial^i\partial^{\beta_3}$) as
	\[
		\partial^{\beta_1} [A(\n v)] = \partial^{\beta_3} \sum_{j\leq n} [A_j(\n v) v_{ij}].
	\]
	The function $p \mapsto A(p)$ is analytic, and so all the derivatives$(D^m A) (\nabla v)$ are bounded over $B_s$. By applying this computation repeatedly, we see that $\partial^{\beta_1} [A(\n v)]$ is a sum of products of derivatives of $v$, multiplied by one derivative in the tensor $(D^m A) (\nabla v)$, with the following properties: at most $|\beta_1|$ factors are present, and each factor has at most one derivative of $v$ is in a direction other than in the multiindex $\beta_1$. In particular, we can bound
	\[
		|\partial^{\beta_1} [A(\n v)]| \leq C(|\beta_1|) \sum_{\zeta_1 + \ldots + \zeta_m = \beta_1, |\zeta_j| \geq 1 }\prod_{j = 1}^m |\partial^{\zeta_j} \nabla v|.
	\]

	Fix $p > 2$, and estimate
	\begin{align*}
		\|\partial^{\beta_1}[A(\n v)] \partial^{\beta_2}w\|_{L^p(B_s)} & \leq \|\partial^{\beta_1}[A(\n v)]\|_{L^{2p}(B_s)}\|\partial^{\beta_2}w\|_{L^{2p}(\alpha,B_s)}\\
		& \leq C_{2p, k -2}\|\partial^{\beta_1}[A(\n v)]\|_{L^{2p}(\alpha,B_s)} \\
		& \leq C \sum_{\zeta_1 + \ldots + \zeta_m = \beta_1, |\zeta_j| \geq 1 } \prod \|\partial^{\zeta_j}\nabla v\|_{L^{2mp}}(\alpha,B_s) \\
		& \leq C \prod_{j = 1}^m C_{2mp, |\zeta_j|} \\
		& \leq C
	\end{align*}
	from our inductive hypothesis used repeatedly, so long as $s \leq s(2kp, k - 2)$. We have shown that
	\[
		\text{div} \ \left[(h'(y_n))^2 A(\n v) \nabla \partial^\alpha w \right] = \text{div} \ \bm{G}
	\]
	with $\|\bm{G}\|_{L^p(\alpha, B_s)} \leq C$. Applying Theorem \ref{t:peral} and using that $\| A - I \| \leq \epsilon$, this gives (with possibly smaller $s$) that
	\[
		\|\nabla \partial^\alpha w\|_{L^p(\alpha, B_{s/2})} \leq C \|\bm{G}\|_{L^p(\alpha,B_s)} \leq C.
	\]
    Now applying Proposition \ref{p:sobolev2}, we obtain the $L^p$ regularity without the weight for one less derivative. 
	This is true for any tangential multiindex $\alpha$ with $|\alpha|\leq k - 1$ and any $i = 1, \ldots, n - 1$, so for any $\alpha' = \alpha + e_i$
	\[
		\|\nabla \partial^{\alpha'} v\|_{L^p(\alpha,B_{s/2})} \leq C.
	\]
	Summing over $\alpha'$ completes the inductive step.  Now applying Proposition \ref{p:sobolev2}, we obtain the $L^p$ regularity without the weight for one less derivative. Thus, we conclude that 
    \[
		\|\nabla \partial^{\alpha'} v\|_{L^p(B_{s/2})} \leq C.
	\]
    Applying the usual Sobolev embedding theorem (without any weights), we conclude that $\partial^{\alpha'} v$ exists and is continuous. This proves Theorem \ref{t:cinfty}. 
    
\end{proof}

\section{$W^{1,p}$ estimates} \label{s:peral}

In this section we show how the technique of Caffarelli-Peral \cite{cp98} can be adapted to give $C^{\infty}$ regularity of the free boundary. We will adapt the presentation given in \cite{b05} which accommodates a nonhomogeneous equation. We note estimates of this type have already been obtained in \cite{dp23} for degenerate elliptic equations when $\alpha(x_n)=x_n^{\tau}$ for some $\tau \in (-1,\infty)$. Since this setting does not include all $\alpha(x_n)$ which we consider, we provide here the proofs. 

The estimates needed are 
\begin{itemize}
 \item Suitable $W^{1,p}$ estimates for an approximating equation
 \item Hardy-Littlewood maximal function bounds in a measure space with a doubling measure. 
\end{itemize}

We remark that from the fourth and fifth lines of \eqref{e:hprop} that $h(t)$ has the doubling property. From the second line of \eqref{e:hprop}, we have that $h'(t)$ inherits the doubling property from $h(t)$ in the following way 
\[
h'(2t)\leq C_1 \frac{h(2t)}{2t} \leq C_2 \frac{h(t)}{t} \leq C_3 h'(t).
\]
Then $d \mu= \alpha(x_n) \ dx=[h'(x_n)]^2 \ dx$ is a doubling measure. 

We consider even solutions $w \in H^{1,2}(\alpha,\Omega)$ over a bounded domain $U$ to 
\begin{equation} \label{e:rhslp}
\int_{\Omega} \langle \alpha(x_n) A(x',x_n) \nabla w, \nabla v \rangle 
= -\int_{\Omega} \alpha(x_n)\langle \bm{G}, \nabla v \rangle,
\end{equation}
for all $v \in H_0^{1,2}(\alpha,\Omega)$. We will assume that $A$ is uniformly elliptic. 
In this section we will follow the presentation in \cite{b05} to prove 
\begin{theorem} \label{t:peral}
 Let $\bm{G} \in L^{p}$ for some $1<p<\infty$, and let $w$ be a solution to \eqref{e:rhslp}. There exists $\delta>0$ depending on $p,\alpha$ such that if $\| A - Id \| \leq \delta$, then for any compact $V \Subset \Omega$, there exists a constant $C$ depending on $V,\Omega,n,\alpha, p$ such that 
 \[
 \| \nabla w\|_{L^p(\alpha,V)} \leq C( \| w \|_{H^{1,2}(\alpha,\Omega)} + \| \bm{G}\|_{L^p(\alpha,\Omega)}). 
 \]
\end{theorem}

For notational convenience, let us define $d \mu=\alpha(x_n)  dx$. 
\begin{lemma} \label{l:smallrhs}
 For any $\epsilon>0$ small, there exists a constant $C(n)$ only depending on dimension $n$, and not on $\alpha$, such that for any weak solution $w$ of \eqref{e:rhslp} in $B_4$ satisfying 
 \begin{equation} \label{e:smallrhs}
  \frac{1}{\mu(B_4)} \int_{B_4} |\nabla w|^2 d \mu \leq 1, \qquad
  \frac{1}{\mu(B_4)} \int_{B_4} |\bm{G}|^2  d \mu \leq \epsilon^2, \qquad \| A - I \|_{L^{\infty}} \leq \epsilon^2
 \end{equation}
 then 
 \begin{equation} \label{e:2bound}
  \frac{1}{\mu(B_4)}\int_{B_4} |w-v|^2 + |\nabla (w-v)|^2 \ d \mu\leq C \epsilon^2,
 \end{equation}
 where $v$ is the weak solution to 
 \[
 \int_{B_4} \langle \nabla v, \nabla \phi \rangle d \mu = 0 \quad \text{ for all } \phi \in H_0^1(\alpha,B_4). 
 \]
\end{lemma}

\begin{proof}
    Let $\tilde{w}$ solve 
    \[
    \int_{B_4} \langle A \nabla \tilde{w}, \nabla \phi \rangle d \mu =0,
    \]
    with $\tilde{w}-w \in H_0^1(\alpha,B_4)$. 
    Then 
    \[
    \begin{aligned}
    \frac{1}{\mu(B_4)}\int_{B_4} |\nabla (\tilde{w}-w)|^2 d \mu &\leq 2  
    \frac{1}{\mu(B_4)}\int_{B_4} \langle A \nabla(\tilde{w}-w), \nabla  (\tilde{w}-w) \rangle\ d \mu \\
    &=-2 \frac{1}{\mu(B_4)}\int_{B_4}  \bm{G} \cdot \nabla (\tilde{w}-w) \ d \mu 
    \leq 2\epsilon 
    \left(\frac{1}{\mu(B_4)}\int_{B_4} |\nabla(\tilde{w}-w)|^2 \ d\mu\right)^{1/2}. 
    \end{aligned}
    \]
    From the above inequality as well as the first inequality in \eqref{e:smallrhs} we have 
    \begin{equation} \label{e:almostone}
      \left( \frac{1}{\mu(B_4)}\int_{B_4} |\nabla \tilde{w}|^2 d \mu \right)^{1/2} \leq 1+2\epsilon. 
    \end{equation}
    Now let $v$ solve 
    \[
    \int_{B_4} \langle \nabla v, \nabla \phi \rangle d \mu=0,
    \]
    with $v -\tilde{w} \in H_0^1(\alpha,B_4)$. 
    Then 
    \[
    \begin{aligned}
        \frac{1}{\mu(B_4)}\int_{B_4}|\nabla (\tilde{w}-v)|^2 \ d\mu&= 
        \frac{1}{\mu(B_4)}\int_{B_4} \langle \nabla (\tilde{w}-v), \nabla (\tilde{w}-v)\rangle \ d\mu \\
        &=\frac{1}{\mu(B_4)}\int_{B_4 }\langle \nabla \tilde{w}, \nabla (\tilde{w}-v)\rangle \ d\mu \\
        &=\frac{1}{\mu(B_4)} \int_{B_4} \langle  (I-A)\nabla \tilde{w}, \nabla (\tilde{w}-v)\rangle \ d\mu \\
        &\leq  \epsilon^2  \left( \frac{1}{\mu(B_4)}\int_{B_4} |\nabla \tilde{w}|^2 d \mu \right)^{1/2}
        \left(\frac{1}{\mu(B_4)}\int_{B_4} |\nabla (\tilde{w}-v)|^2 \ d\mu \right)^{1/2}\\
        &\leq \epsilon^2(1+2\epsilon)  \left(\frac{1}{\mu(B_4)}\int_{B_4} |\nabla (\tilde{w}-v)|^2 \ d\mu \right)^{1/2},
    \end{aligned}
    \]
    with the last inequality coming from \eqref{e:almostone}. 
    From the triangle inequality we obtain \eqref{e:2bound} for the gradient term. The bound for the function term follows from \eqref{e:equalsobolev} with $p=2$.  
\end{proof}

We now state some essential tools for the Caffarelli-Peral method. 
The maximal function is defined as
\[
(\mathcal{M}g)(x)=\sup_{r>0} \frac{1}{\mu(B_r)}\int_{B_r(x)} |g| \ d \mu.
\]
Since $\mu$ is a doubling measure, we have the Vitali covering lemma, from which   we obtain maximal function theorems. (see Theorem 2.2 in \cite{h01}). 
\begin{theorem} \label{t:hlm}
For all $t>0$ we have the weak 1-1 estimate
\[
\mu(\{(\mathcal{M}g)>t\}) \leq \frac{C_1}{t} \int_{\mathbb{R}^n} |g|  \ d \mu
\]
and the $L^p$ estimate
\[
\int_{\mathbb{R}^n} |(\mathcal{M}g)|^p d \mu \leq C_p \int_{\mathbb{R}^n} |g|^p \ d  \mu
\]
where the constants $C_1$ and $C_p$ only depend on the doubling constant for $\mu$. 
\end{theorem}

\begin{lemma} \label{l:levelset}
 There is a constant $N_1$ so that for any $\eta>0$ there exists a small $\epsilon>0$, and if $w$ is a weak solution of \eqref{e:rhslp} in $\Omega \supset B_6$ with 
 \begin{equation} \label{e:N1}
 B_1 \cap \{x \in \Omega: \mathcal{M}(|\nabla w|^2(x))\leq 1\} 
 \cap \{x \in \Omega : \mathcal{M}(|\bm{G}|^2)(x) \leq \epsilon^2\} \neq \emptyset
 \end{equation}
 and $\| A - I \| \leq \epsilon^2$, then 
 \[
 \mu(\{x \in \Omega: \mathcal{M}(|\nabla w|^2)(x)>N_1^2\} \cap B_1)< \eta \mu(B_1). 
 \]
\end{lemma}

\begin{proof}
 From \eqref{e:N1}, there exists a point $x^0 \in B_1$ satisfying 
 \begin{equation} \label{e:N1a}
  \frac{1}{\mu(B_r)}\int_{B_r(x^0) \cap \Omega} |\nabla w|^2 \ d \mu \leq 1, 
  \quad 
  \frac{1}{\mu(B_r)} \int_{B_r(x^0) \cap \Omega} |\bm{G}|^2 \ d \mu \leq \epsilon^2 \quad \text{ for all } r>0.
 \end{equation}
 Since $B_4(0)\subset B_5(x^0)$, we have 
 \[
 \frac{1}{\mu(B_4)}\int_{B_4} |\bm{G}|^2 \ d \mu 
 \leq \frac{\mu(B_5(x^0))}{\mu(B_4)}\frac{1}{\mu(B_5(x^0))}
 \int_{B_5(x_0)} |\bm{G}|^2  \ d\mu \leq C(n,\alpha) \epsilon^2.
 \]
 Similarly, 
 \[
 \frac{1}{\mu(B_4)}\int_{B_4}|\nabla w|^2 \ d\mu \leq C(n,\alpha). 
 \]
 Dividing $u$ and $\bm{G}$ by $C(n,\alpha)$, we can apply Lemma \ref{l:smallrhs} to the $\alpha$-harmonic replacement $v$. From Lemma \ref{l:gradapprox} we have 
 \[
 \| \nabla v \|_{L^{\infty}(B_3)}^2 \leq N_0^2
 \]
 with $N_0$ depending on $n$ and $\alpha$. 
 We now claim 
 \begin{equation} \label{e:N1claim}
 \{x : \mathcal{M}(|\nabla w|^2)>N_1^2\} \cap B_1 
 \subset \{x : \mathcal{M}(|\nabla w|^2 \chi_{B_4})>N_0^2\}\cap B_1,
 \end{equation}
 where $N_1^2=\max{4N_0^2,C}$ with $C$ only depending on $n,\alpha$. To show this, suppose that 
 \[
 x^1\in \{x : \mathcal{M}(|\nabla w|^2 \chi_{B_4})>N_0^2\}\cap B_1.
 \]
 For $r\leq 2$, $B_r(x^1) \subset B_3$, so that 
 \[
 \frac{1}{\mu(B_r(x^1))}\int_{B_r(x^1)} |\nabla w|^2 \ d\mu
 \leq \frac{2}{\mu(B_r(x^1))} \int_{B_r(x^1)} 
 |\nabla (w-v)|^2 + |\nabla v|^2 \ d \mu \leq 4N_0^2. 
 \]
 Now for $r>2$, $x^0 \in B_r(x^1)\subset B_{2r}(x^0)$ so by 
 \eqref{e:N1a}
 \[
 \frac{1}{\mu(B_r(x^1))} \int_{B_r \cap \Omega}|\nabla w|^2 \ d\mu
 \leq \frac{\mu(B_{2r}(x^0))}{\mu(B_r(x^1))} 
 \frac{1}{\mu(B_{2r}(x^0))} \int_{B_{2r}(x^0)\cap \Omega}
 |\nabla w|^2 \ d\mu \leq C. 
 \]
 then the above two inequalities show that 
 \[
  x^1 \in \{x : \mathcal{M}(|\nabla w|^2)\leq N_1^2\} \cap B_1,
 \]
 and our claim in \eqref{e:N1claim} is shown. By \eqref{e:N1claim}, the weak 1-1 estimate in Theorem \ref{t:hlm}, and inequality   \eqref{e:2bound} we have 
 \[
 \begin{aligned}
 \mu(\{x: \mathcal{M}(|\nabla w|^2)(x)>N_1^2\}\cap B_1)
 &\leq \mu(\{x: \mathcal{M}(|\nabla (w-v|^2\chi_{B_4})(x)(x)>N_0^2\}\cap B_1)\\
 &\leq \frac{C}{N_0^2} \int_{B_4} |\nabla (w-v)|^2 \ d\mu\ \leq \frac{C}{N_0^2} \epsilon^2 =\eta \mu(B_1).
 \end{aligned}
 \]  
\end{proof}

We now wish to apply Lemma \ref{l:levelset} to balls of various radii and translations off the thin space $\{x_n=0\}$. 

\begin{lemma} \label{l:scale}
  Assume that $w$ is a weak solution of \eqref{e:rhslp} in $\Omega \supset B_{6r}(x^0)$. If 
 \[
 \mu(\{x \in \Omega: \mathcal{M}(|\nabla w|^2)(x)>N_1^2\}\cap B_{r}(x^0))\geq \eta \mu(B_r(x^0),
 \]
 then 
 \[
 B_r(x^0) \subset \{x \in \Omega: \mathcal{M}(|\nabla w|^2)(x)>1\} \cup \{x \in \Omega: \mathcal{M}(|\bm{G}|^2)(x)>\epsilon^2\}.
 \]
\end{lemma}

\begin{proof}
 If $x^0$ is centered on the thin space, then the Lemma follows directly from rescaling. Suppose now that $B_r(x^0)\subset B_{2r}(x^1)$ with $x^1$ centered on the thin space. Suppose now 
 \[
  x^2 \in B_{B_r(x^0)} \cap \{x \in \Omega: \mathcal{M}(|\nabla w|^2(x))\leq 1\} 
 \cap \{x \in \Omega : \mathcal{M}(|\bm{G}|^2)(x) \leq \epsilon^2\}.
 \]
 Since the Lemma holds for $B_{2r}(x^1)$, we have
 \[
 \begin{aligned}
  &\mu(\{x \in \Omega: \mathcal{M}(|\nabla w|^2)(x)>N_1^2\} \cap B_r(x^0)) \\
  &\quad \subset \mu(\{x \in \Omega: \mathcal{M}(|\nabla w|^2)(x)>N_1^2\} \cap B_{2r}(x^1))
  \leq \eta \mu(B_2) \leq C \eta \mu(B_1(x^0)). 
 \end{aligned}
 \]
 If $B_{2r}(x^0)\cap \{x_n =0\}=\emptyset$, then the proof of Lemma \ref{l:levelset} goes through exactly as before for the rescaled function since the approximating $\alpha$-harmonic function has interior $C^{1,1}$ estimates from interior regularity. Then rescaling backwards we obtain the Lemma for all balls. 
\end{proof}

Using induction we  obtain the following corollary. 
\begin{corollary} \label{c:push}
 Suppose that $w$ is a weak solution to \eqref{e:rhslp} in a domain $\Omega \supset B_{6}$. Assume that
 \[
 |\{x \in \Omega: \mathcal{M}(|\nabla w|^2)(x)>N_1^2\}|< \eta \mu(B_1).
 \]
 Let $k \in \mathbb{N}$ and set $\eta_1 = C_1 \eta$. Then 
 \[
 \begin{aligned}
 &\mu(\{x \in \Omega: \mathcal{M}(|\nabla w|^2)(x)>N_1^{2k})\\
 & \ \leq \sum_{i=1}^k \eta_1^i \mu(\{x \in B_1: \mathcal{M}(|\bm{G}|^2)>\epsilon^2 N_1^{2(k-i)}\})+ \eta_1^k \mu(\{x \in B_1: \mathcal{M}(|\nabla w|^2)>1\})
 \end{aligned}
 \]
\end{corollary}

\begin{proof}
 We proceed by induction. For the base case $k=1$ we consider 
 \[
 \begin{aligned}
   &S_1:=\{x \in B_1 : \mathcal{M}(|\nabla w|^2)(x)>N_1^2\}, \\
   &S_2:=\{x \in B_1 : \mathcal{M}(|\bm{G}|^2)(x)>\epsilon^2\}
   \cup \{x \in B_1 : \mathcal{M}(|\nabla w|^2)>1\}.
 \end{aligned}
 \]
 From assumption we have $\mu(S_1)<\eta \mu(B_1)$. From Lemma \ref{l:scale} we have that whenever $\mu(S_1 \cap B_r(x))\geq \eta \mu(B_r(x))$ with $0<r<1$, then 
 \[
 B_r(x)\cap B_1 \subset S_2. 
 \]
 From the Vitali covering lemma (see \cite{h01}) we have 
 \[
 |S_1| \leq C_1 \eta|S_2|,
 \]
 with the constant only depending on the doubling constant of $\alpha$. 
 We now assume that the conclusion is valid for some $k>1$. 
 Let $w_1 =w/N_1$ and $\bm{G}_1 = \bm{G}/N_1$. Then with $w_1$ as a weak solution we obtain 
 \[
 \mu(\{x\in \Omega : \mathcal{M}(|\nabla w_1)|^2)(x)>N_1^2\})<\eta \mu(B_1). 
 \]
 Then by the induction assumption, we have 
 \[
 \begin{aligned}
     &\mu(\{x \in \Omega: \mathcal{M}(|\nabla w|^2)(x)>N_1^{2(k+1)}\})
     =\mu(\{x \in \Omega: \mathcal{M}(|\nabla w|^2)(x)>N_1^{2k}\})\\
     &\leq \sum_{i=1}^k\eta_1^i\mu(\{x \in B_1: \mathcal{M}(|\bm{G}|^2)(x)>\epsilon^2 N_1^{2(k-i)}\}) 
      +\eta_1^k\mu(\{x \in B_1: \mathcal{M}(|\nabla w_1|^2)(x)>1\})\\
     &\leq \sum_{i=1}^{k+1}\eta_1^i\mu(\{x \in B_1: \mathcal{M}(|\bm{G}|^2)(x)>\epsilon^2 N_1^{2(k+1-i)}\})  +\eta_1^{k+1}\mu(\{x \in B_1: \mathcal{M}(|\nabla w|^2)(x)>1\}),
 \end{aligned}
 \]
 and the conclusion is proven. 
\end{proof}

We now prove Theorem \ref{t:peral}.
\begin{proof}
 By multiplying our solution $w$ by a small constant, we can assume $\| \bm{G}\|_{L^p(B_6)}$ is small and that 
 \[
 \mu(x \in B_6: \mathcal{M}(|\nabla w|^2)(x)>N_1^2\cap B_1)<\eta\mu(B_1). 
 \]
 Since $\bm{G} \in L^p(B_6)$ we have from the strong $p$-$p$ estimates in Theorem \ref{t:hlm} that $\mathcal{M}(|\bm{G}|^2)\in L^{p/2}(B_6)$. We now use a standard tool from measure theory (see Lemma 7.3 in \cite{cc95}) to see there is a constant $C$ depending on $\alpha, \epsilon, p, N_1$ such that 
 \[
 \sum_{k=0}^{\infty} N_1^{pk}\mu(\{x\in B_1: \mathcal{M}(|\bm{G}|^2)(x)>\epsilon^2 N_1^{2k}\})\leq C \|\mathcal{M}(|\bm{G}|^2) \|_{L^{p/2}(B_6)}^{p/2}.
 \]
 Using the strong $p$-$p$ estimates on the right hand side above as well as the smallness condition on $\| \bm g\|_{L^p}$ we have 
 \[
 \sum_{k=0}^{\infty} N_1^{pk}\mu(\{x\in B_1: \mathcal{M}(|\bm{G}|^2)(x)>\epsilon^2 N_1^{2k}\})\leq 1. 
 \]
 We now utilize $p>2$ in the following computation.
 \[
 \begin{aligned}
 &\sum_{k=0}^{\infty} N_1^{pk} \mu(x \in B_1: \mathcal{M}(|\nabla w|^2)(x)>N_1^{2k}) \\
 &\leq \sum_{k=1}^{\infty} N_1^{pk} \bigg(\sum_{i=1}^k \eta_1^i\mu(\{x\in B_1: \mathcal{M}(|\bm{G}|^2)(x)>\epsilon^2 N_1^{2(k-i)}\})\\
 &\quad +\eta_1^k \mu(\{x\in B_1: \mathcal{M}(|\nabla w|^2)(x)>1\})\bigg) \\
 &=\sum_{i=1}^{\infty}(N_1^p \eta_1)^i \bigg(\sum_{k=i}^{\infty} N_1^{p(k-i)} \mu(\{x \in B_1: \mathcal{M}(|\bm{G}|^2)(x)>\epsilon^2 N_1^{2(k-i)}\}) \bigg)\\
 &\quad + \sum_{k=1}^{\infty} (N_1^p\eta_1)^k\mu(\{x\in B_1: \mathcal{M}(|\nabla w|^2)(x)>1 \}) \\
 &\leq C \sum_{k=1}^{\infty}(N_1^p \eta_1^k)\\
 &< +\infty. 
 \end{aligned}
 \]
 Thus using Lemma 7.3 in \cite{cc95} we have that 
 $\mathcal{M}(|\nabla w|^2) \in L^{p/2}(B_6)$. Therefore, $\nabla w \in L^p(B_6)$. 
\end{proof}

 \section*{Acknowledgements}
This project was initiated  while the authors stayed at Institute Mittag Leffler (Sweden), during the program Geometric aspects of nonlinear PDE. The authors wish to thank Georg Weiss for helpful conversations on this topic. H. Shahgholian was supported by Swedish research Council. D. Kriventsov was supported by NSF DMS grant 2247096.
The authors are grateful to the reviewer for their careful reading and constructive suggestions.

\section*{Declarations}

\noindent {\bf  Data availability statement:} All data needed are contained in the manuscript.

\medskip

\noindent {\bf  Funding and/or Conflicts of interests/Competing interests:} The authors declare that there are no financial, competing or conflict of interests.

\bibliographystyle{plain}
\bibliography{bibliography.bib}

\end{document}